\documentclass[11pt]{amsart}

\usepackage{color}

\usepackage{lineno, subcaption}

\usepackage{tikz, tikz-cd}
\usetikzlibrary{arrows,shapes,calc}
\usetikzlibrary{decorations.pathreplacing,decorations.markings, decorations.pathmorphing}
\usepackage{graphicx}

\usepackage{latexsym}
\usepackage{amsthm, amssymb,amsfonts,mathrsfs,amsxtra,dsfont}
\usepackage{amsmath, amscd}
\usepackage{bm}
\usepackage{mathtools}
\usepackage{euscript}

\usepackage{scalerel}

\usepackage[colorinlistoftodos]{todonotes}

\newcommand{\svsp}[1]{ } 

\numberwithin{equation}{section}

\DeclareMathAlphabet{\mathitbf}{OT1}{cmr}{bx}{it}
\DeclareMathAlphabet{\mathsfbf}{OT1}{cmss}{bx}{n}

\definecolor{pptbl}{RGB}{79,129,189}

\definecolor{light-gray}{gray}{0.90}
\definecolor{med-gray}{gray}{0.70}

  \def \begrm{ \begin{description} \item[{\mathbf Remark:\hspace{2.5em} }]  }
  \def \endrm{ 
		  \end{description} }

    \def \be { \begin{equation}   }
    \def \ee { \end{equation}  }

    \def \bea { \begin{equation}   \begin{array}{ c l}
	       }
    \def \eea { \end{array} \end{equation}  }

    \def \bma { \[ \begin{array}{ c l}
	       }
    \def \ema {    \end{array}  \] }

\newcommand{\Nosoln}[2] {}

\newcommand{\hs }[1]{ \h_{\mathrm{#1}} } 

    \def \begpf{ \begin{description} \item[\mbox{$\mathbf Proof:$\hspace{2.5em} }]  }
    \def \endpf{ \begin{flushleft} $\Box $ \end{flushleft}   \end{description} }

    \def \begrm{ \begin{description} \item[{\mathbf Remark:\hspace{2.5em} }]  }
    \def \endrm{ 
		  \end{description} }

    \def \be { \begin{equation}   }
    \def \ee { \end{equation}  }

    \def \bma { \[ \begin{array}{ c l}
	       }
    \def \ema {    \end{array}  \] }

\usepackage{hyperref}
\usepackage{enumerate}
\usepackage{cleveref}

\DeclareMathOperator{\Hom}{Hom}
\DeclareMathOperator{\id}{id}

\DeclareMathOperator*{\argmin}{arg\,min}

\newcommand{\dist}{d}
\DeclareMathOperator{\proj}{proj}
\DeclareMathOperator*{\colim}{colim~}

\newcommand{\exc}{\mathcal E}
\newcommand{\infexc}{\exc^{\infty}}

\newcommand{\im}{\mathrm{im}}
\newcommand{\conley}{\mathcal Ch}

\newcommand{\conleyet}{\conley^{\varepsilon,T}}

\newcommand{\src}{\mathfrak{s}}
\newcommand{\tgt}{\mathfrak{t}}
\newcommand{\graph}[1]{{{#1}}}

\newcommand{\init}[1]{{#1}_{\mathbf i}}
\newcommand{\fin}[1]{{#1}_{\mathbf f}}

\newcommand{\one}{\bm{1}}

\newcommand{\stp}{\mathrm{stop}}

\newcommand{\catname}[1]{\bm{\mathrm{#1}}}
\newcommand{\sset}{\catname{S}}

\renewcommand{\hs}{\catname{H}}
\newcommand{\hempty}{\hs_\emptyset}
\newcommand{\dirhs}{\catname{DH}}
\newcommand{\hsdn}{\hs_{D\!N}}

\newcommand{\mfld}{\catname{M}}

\newcommand{\hsem}{\hs_{\mathrm{em}}}

\newcommand{\cs}{\catname{D}}
\newcommand{\dgraph}{\catname{G}}

\newcommand{\sets}{\catname{Set}}
\newcommand{\tops}{\catname{Top}}

\newcommand{\hop}{\mathrm{hop}}

\theoremstyle{plain}
\newtheorem{thm}{Theorem}[section]
\newtheorem{prop}[thm]{Proposition}
\newtheorem{lemma}[thm]{Lemma}
\newtheorem{cor}[thm]{Corollary}

\theoremstyle{definition}
\newtheorem{defn}[thm]{Definition}
\newtheorem{remark}[thm]{Remark}
\newtheorem{example}[thm]{Example}

\tikzset{endpoint/.style={circle,fill=black,draw,minimum size=4pt,inner sep=0pt},
guard/.style={circle,fill=red,draw,minimum size=4pt,inner sep=0pt},
mid arrows/.style={postaction={decorate,decoration={
        markings,
        mark=at position .25 with {\arrow[#1]{stealth}},
        mark=at position .75 with {\arrow[#1]{stealth}}
      }}},
directed arrow/.style={decorate,decoration={snake, segment length=6mm}}
}

\author[J. Culbertson]{Jared Culbertson}
\address{Autonomy Capabilities Team, Air Force Research Laboratory, Dayton, OH 45433, USA}
\email{jared.culbertson@us.af.mil}

\author[P. Gustafson]{Paul Gustafson}
\address{Department of Electrical Engineering, Wright State University, Dayton, OH 45435, USA}
\email{paul.gustafson@wright.edu}

\author[D.E. Koditschek]{Daniel E. Koditschek}
\address{School of Engineering and Applied Science, University of Pennsylvania, Philadelphia, PA 19104, USA}
\email{kod@seas.upenn.edu}

\author[P.F. Stiller]{Peter F. Stiller}
\address{Department of Mathematics, Texas A\&M University, College Station, TX, 77840, USA}
\email{stiller@math.tamu.edu}

\begin{document}

\title[Formal composition of hybrid systems]{Formal composition of hybrid systems}

\maketitle

\begin{abstract}
We develop a compositional framework for formal synthesis of hybrid systems using the language of category theory.  More specifically, we provide mutually compatible tools for hierarchical, sequential, and independent parallel composition.    In our framework, hierarchies of hybrid systems correspond to template-anchor pairs, which we model as spans of subdividing and embedding semiconjugacies.  Hierarchical composition of template-anchor pairs corresponds to the composition of spans via fiber product.  To model sequential composition, we introduce ``directed hybrid systems,'' each of which flows from an initial subsystem to a final subsystem in a Conley-theoretic sense.  Sequential composition of directed systems is given by a pushout of graph embeddings, rewriting the continuous dynamics of the overlapping subsystem to prioritize the second directed system. Independent parallel composition corresponds to a categorical product with respect to semiconjugacy.  To formalize the compatibility of these three types of composition, we construct a vertically cartesian double category of hybrid systems where the vertical morphisms are semiconjugacies, and the horizontal morphisms are directed hybrid systems.
\end{abstract}

\section{Introduction}

We aim to construct a physically-grounded compositional framework for hybrid system synthesis, particularly targeted at applications in robotics. Compositionality lies at the heart of language in general \cite{werning2012oxford} and its formalization underlies much of computer science in particular \cite{Leeuwen_1990}. However the behavioral modularization of physical  synthesis for digital computing that arguably ushered in the information technology revolution \cite{Mead_Conway_1980} has proven much harder to achieve in analog computing technology \cite{Mead_1989}. There are fundamental reasons for this challenge to become more severe in machines  intended to perform mechanical work on their environments \cite{Whitney_1996}.

\subsection{Motivation}
Our formalism is motivated by three distinct notions of behavioral composition that have emerged over the past thirty years in the robotics literature: sequential, hierarchical and parallel. We are specifically focused on versions of these constructions championed by the third author and collaborators that afford simultaneously robust physical realization in working robots as well as  formal proofs of correctness relative to empirically effective mathematical models of the component hardware. {\em Sequential composition} \cite{Burridge_Rizzi_Koditschek_1999} formalizes (for systems undergoing energetic exchange with their environment) notions of ``pre-image back-chaining" \cite{Lozano-Perez_Mason_Taylor_1984} originating in some of the earliest AI planning literature \cite{Fikes_Nilsson_1971}.  Such constructions have earned wide attention in robotics   \cite[Ch. 8.5]{LaValle_2006} as they correspond to broadly useful event-based concatenation of behaviors over time: follow one control policy until reaching an appropriately guarded state, then follow another. Our notion of { \em hierarchical composition} has a still  older pedigree, based on the folklore dynamical systems ``collapse of dimension" concept so deeply engrained in the literature as to appear in even the most elementary texts \cite[Ch.3.5]{Strogatz_1994}. Addressing the longstanding ``degrees of freedom" problem \cite{Bernstein_1967} in such terms at once affords an organizational framework for analysis of animal motor activity \cite{Full_Koditschek_1999} and synthesis of robot controllers \cite{Rizzi_Koditschek_1994,Saranli_Schwind_Koditschek_1998}, earning the notion an enduring following in both neuromechanics \cite{Nishikawa_Biewener_Aerts_Ahn_Chiel_Daley_Daniel_Full_Hale_Hedrick_et_al_2007} and robotics \cite{Yang_Bellingham_Dupont_Fischer_Floridi_Full_Jacobstein_Kumar_McNutt_Merrifield_et_al_2018}. Finally, { \em parallel composition}, simultaneous operation of distinct behaviors in the same body, while evidently useful, has only relatively  recently been achieved in an empirically reliable form relevant to highly energetic mechanisms \cite{Raibert_1986} (with a corresponding mathematical theory only now beginning to emerge \cite{De_Koditschek_2015,De_Koditschek_2018}), essentially due to the challenges of circumventing destabilizing ``cross-talk'' outlined in \cite{Whitney_1996}.  

Because robotics applications inevitably incur sudden transitions between dynamics and state spaces consequent upon the making and breaking of different contacts with different portions of the environment, our constructions must depart from classical  theory to embrace a notion of hybrid dynamical systems. We are primarily interested in a formalism focused on  {\em  non-blocking }\cite[Def.III.1]{lygeros2003dynamical} and { \em deterministic} \cite[Def.III.2]{lygeros2003dynamical} executions --- the hybrid version of existence and uniqueness properties familiar from classical dynamical systems theory. Hence, our constructions are roughly guided by a simple and  reasonably physically realistic model of robot mechanics that assures these properties \cite{Johnson_Burden_Koditschek_2016}.  
\subsection{Contributions}
We formalize a useful subset of each of these three motivating notions using category theory, which is particularly well-suited for describing composition and abstraction.  More importantly, the Curry--Howard correspondence provides a pathway for translating the categorical results developed in this paper to the setting of functional programming.  Although we will not explore connections to robotic behavioral programming in the current paper, our longer-term ambitions also involve developing more physically grounded analogues of existing formal approaches to motion planning using linear-time temporal logic \cite{kress2009temporal} and functional reactive programming \cite{hudak2002arrows}.  Similarly, we do not address the construction of exponential objects for hybrid systems categories, a fundamentally important step in constructing a full-fledged functional programming language for hybrid systems requiring more careful consideration. Indeed, in the smooth case the search for ``convenient'' cartesian closed categories of generalized manifolds has been a central line of research in synthetic differential geometry \cite{moerdijk2013models, kock2006synthetic}.    As a first step in this long journey, the main contribution of this paper is the investigation of a series of categories of hybrid systems that are both mathematically rigorous and also faithful to the way that robotics engineers often approach model development. Our investigation culminates in the construction of a double category of hybrid systems supporting forms of hierarchical, sequential, and independent parallel composition.

We were heavily influenced by the recent preprint~\cite{lerman} and the elegant definitions therein of a category of hybrid systems and semiconjugacies---loosely, execution-preserving maps. To accommodate examples from robotics, we develop a modified definition of an automaton-based hybrid system where, instead of manifolds with corners, our continuous modes occupy arbitrary subsets of smooth manifolds. In keeping with this focus on robotics, we restrict resets to be functions rather than the more general set-theoretic relations. We also generalize the notion of hybrid semiconjugacy to allow resets to be sent to ``trivial resets'' ({\em i.e.}, identity maps on continuous modes), which underlies our definition of a subdivision of a hybrid system.

As a first attempt at parallel composition, we focus here on combining decoupled hybrid systems by showing that our category of hybrid systems is cartesian. A similar theorem is proved in \cite{lerman-schmitt:open_hybrid_systems}, where the authors also explore more complex coupled systems using interconnection maps. We focus on the simpler decoupled setting to explore connections between parallel composition and the other two forms of composition under consideration. 

To study hierarchical composition, we first formalize the notion of a template-anchor pair of hybrid systems.  In dynamical systems theory \cite{Full_Koditschek_1999, Kvalheim_Revzen_2016}, a template is a low degree of freedom, idealized model of a physical system.  An anchor corresponding to such a template is a high degree of freedom, more realistic model of the same system.  A template-anchor pair consists of an embedding of a template model as an attracting, invariant subsystem of a corresponding anchor model.   If trajectories in the anchor converge to the template sufficiently quickly, then the dynamics of the template provides a good approximation for the dynamics of the anchor.

In the continuous setting, template-anchor pairs have a beautiful and well-developed theory.  For example, under reasonable assumptions the basin of attraction of the template forms a topological disk bundle, and the transverse dynamics admit a global linearization \cite{eldering2018global}.  Unfortunately, mismatches between resets in templates and anchors complicate the theory in the hybrid setting. 

To address these complications, we introduce formal subdivisions of hybrid  systems---semiconjugacies satisfying a fiber product property for executions, allowing us to add formal resets to template systems. We then define a template-anchor pair as a span for which the left leg is a template, the roof is a subdivision of the template, and the right morphism embeds this subdivision into the anchor as an attracting, invariant subsystem.   Hierarchical composition corresponds to composing spans by taking a fiber product of subdivisions over a system that is both an anchor for a simpler template and a template for a more complicated anchor.

To model sequential composition, we formulate a notion of ``directed hybrid system,'' a system in which a generic execution flows from a domain subsystem to a codomain subsystem.  Intuitively, this formalism gives a basis for modeling simple robotic behaviors as directed hybrid systems which have specified initial and final interfaces available for linking behaviors to achieve more complex behaviors.  Standard notions of generic executions (from almost all initial conditions in a measure-theoretic or topological sense) do not compose well,  so we use an adaptation of Conley's $(\varepsilon,T)$-chains~\cite{conley1978isolated} to the hybrid setting in place of executions.  Interestingly, we prove that there exists a double category of hybrid systems for which directed systems are the horizontal morphisms and semiconjugacies form the vertical morphisms, providing a setting for exploring both model abstraction together with sequential composition.  We also prove a theorem showing the compatibility of independent parallel composition with this notion of sequential composition.

\subsection{Reader's Guide}
We have attempted to make the results in this paper understandable to a diverse audience including (i) category theorists interested in applying categorical ideas to a very concrete setting, (ii) engineers  interested in a formal language for expressing hybrid dynamical concepts, and (iii) roboticists interested in abstractions of behavioral models. For this reason, when given a choice between a sophisticated categorical description and a concrete straightforward description, we prefer to give the straightforward one, though this results in an exposition that is at times less than optimally compact. Moreover, we appreciate that there are unavoidable mathematical abstractions, dynamical intricacies and robotics paradigms in this paper that are not generally presented in concert. Because of this, we have also sought to diligently provide pointers to both the mathematical literature (where broader concepts underlie our constructions) as well as to the robotics literature (where our formalism contacts more sophisticated platforms and models). While some of the technical results may be inaccessible to those without some background in both dynamical systems and category theory, we hope that the main compositional ideas will be of interest more broadly.

On a first reading, we recommend focusing on the relatively self-contained definitions of hybrid systems (\Cref{def:HybridSystem}), hybrid semiconjugacies (\Cref{def:hyb_sem}),  and executions (\Cref{def:execution}) within \Cref{sec:hs}.    \Cref{sec:cat} describes various properties of the category of hybrid systems and semiconjugacies, including the construction of independent parallel compositions (\Cref{prop:cartesian}) and the fibration of hybrid systems over directed graphs (\Cref{prop:fibration}) that provide important tools for the constructions in the subsequent sections. The main ideas of  \Cref{sec:hierarchical,sec:sequential} are basically independent and can be read in either order. \Cref{sec:hierarchical} describes hierarchical composition, which relies on the key notions of attracting sets (\Cref{def:iso_inv}), subdivisions (\Cref{def:hyb_sub}), and template-anchor pairs (\Cref{def:template-anchor}).  The main result of this section is the definition of the hierarchical composition of template-anchor pairs as a fiber product (\Cref{thm:ta_comp}). 
\Cref{sec:sequential} defines a double category for sequential composition based on  $(\varepsilon,T)$-chains (\Cref{def:chain}) which leads to the definition of a directed system (\Cref{def:directed}).  The main theorem verifies the compatibility of sequential composition of directed systems with hybrid semiconjugacy (\Cref{thm:double_cat}).

\section{Hybrid dynamics} \label{sec:hs}
In this section, we define a category $\hs$ of hybrid systems and semiconjugacies.   A hybrid system consists of a directed graph reflecting the discrete dynamics of the system, equipped with a continuous dynamical system for each vertex and a reset map for each edge.  We begin by defining our underlying category of graphs.

\subsection{Graphs}
\label{sec:graphs}
We first set notation and give some elementary properties of the particular flavor of graphs that will provide the substrate for the discrete portion of the dynamics in our treatment of hybrid systems. To internalize subdivisions of systems as morphisms, we use a nonstandard definition of graph morphism, allowing edges to map to vertices.

\begin{defn}
A {\bf directed graph} is a tuple $\graph G = (V, E, \src, \tgt)$, where $V, E$ are sets and $\src, \tgt\colon V \sqcup E \to V$ with $\src(v) = v = \tgt(v)$ for all $v \in V$. A directed graph $\graph G$ is finite if both $V$ and $E$ are finite. Note that we write $V(G), E(G)$ to distinguish vertex and edge sets when dealing with multiple graphs.

We define a {\bf graph morphism} $f\colon \graph G_1 \to \graph G_2$ to be a pair of maps $(f_V, f_E)$ with $f_V\colon V(\graph G_1) \to V(\graph G_2)$ and $f_E\colon E(\graph G_1) \to V(\graph G_2) \sqcup E(\graph G_2)$ satisfying for all $e \in E(\graph G_1)$
\begin{linenomath*}
\begin{align*}
    \src_2(f_E(e)) &= f_V(\src_1(e))\\
    \tgt_2(f_E(e)) &= f_V(\tgt_1(e)).
\end{align*}
\end{linenomath*}
Composition of graph morphisms is done in the obvious way: given $f\colon \graph G \to \graph H$ and $g\colon \graph H \to \graph K$, we can define $g\circ f\colon \graph G \to \graph K$ by setting $(g\circ f)_V = g_V \circ f_V$, $(g \circ f)_E(e) = (g_E \circ f_E)(e)$ when $f_E(e) \in E(\graph H)$, and $(g \circ f)_E(e) = (g_V \circ f_E)(e)$ when $f_E(e) \in V(\graph H)$. It is easy to check that this composition is associative, so that we get a category $\dgraph$ of directed graphs. 
\end{defn}

A $\dgraph$-morphism $G \xrightarrow{f} G'$ is monic precisely if the maps $f_V, f_E$ are injective and $\im(f_E) \subset E(G')$. Similarly, $f$ is an epimorphism precisely if $f_V\colon V(G) \to V(G')$ is surjective and $E(G') \subset \im(f_E)$.

We note that the category $\dgraph$ is equivalent to the category of directed reflexive graphs, which are usually presented by the data $G = (V, E, \src, \tgt, \mathfrak{e})$, where  $\src, \tgt\colon E \to V$ and $\mathfrak{e}\colon V \to E$ with $\src \circ \mathfrak{e} = \id_V = \tgt \circ \mathfrak{e}$. A directed reflexive graph morphism $f\colon G \to G'$ is then more easily specified as a map $f\colon E(G) \to E(G')$ satisfying the obvious compatibility conditions with $\src, \tgt$ and $\mathfrak{e}$. This formulation, however, specifies a distinguished edge for each vertex, which is inconvenient for indexing the reset maps of a hybrid system with the edge set of a graph.  In a few cases, our choice in formalism leads to slightly less elegant definitions, which we hope will not distract the reader from the otherwise straightforward underlying ideas.

\subsection{Hybrid systems}
The next several definitions have their origins in standard approaches to carefully defining hybrid systems such as \cite{simic, lerman, htp:bisimulations}, but have been tailored to the needs of our intended applications. In particular, we have designed the following definitions to be compatible with  (hierarchical compositions of) template and anchor pairs, as studied in \Cref{sec:templates_anchors}. 

For the remainder of this paper, the term {\bf manifold} will refer to a smooth manifold without boundary. We will denote the category of manifolds and smooth maps by $\mfld$. 
\begin{defn} \label{def:HybridSystem}
A {\bf hybrid system} $H$ consists of
\begin{enumerate}[(1)]
	\item a directed graph $G = (V, E, \src, \tgt)$;   
	\item for each {\bf continuous mode} $v \in V$, 
	\begin{itemize}
	    \item an {\bf ambient manifold} $M_v$
	    \item a vector field $X_v$ on $M_v$
	    \item an {\bf active set} $I_v \subset M_v$  
	    \item a {\bf flow set} $F_v \subset I_v$
	\end{itemize}
	\item for each {\bf reset} $e \in E$, a {\bf guard set} $Z_e \subset I_{\src(e)}$ and an associated {\bf reset map}\footnote{For the sake of simplicity, we do not make any continuity assumptions on the reset maps.  However, our constructions also apply to hybrid systems whose reset maps are required to be continuous or smooth (where smoothness is defined as in \Cref{sec:smooth}).} $r_e\colon Z_e \to I_{\tgt{(e)}}$.
\end{enumerate}
\end{defn}

As we define precisely below, all executions of a hybrid system take place in the active sets.  The flow sets correspond to the regions for which executions may follow the vector fields.  If an execution hits a guard set, it may take the corresponding reset map.

We will write $G(H)$, $V(H)$, and $E(H)$ for the graph, vertex set, and edge set of an arbitrary hybrid system $H$. It will sometimes be convenient to refer to the union of all guards in a given mode $v \in V$, which we will denote by 
\[
    Z_v = \bigcup_{\src(e) = v} Z_e. 
\]
Generally we reserve the following symbols throughout the paper: $M$ will refer to manifolds, $X$ to vector fields, $G$ to graphs, $Z$ to guard sets, $F$ to flow sets, and $I$ to active sets. If there is potential ambiguity, we will add superscripts ({\em e.g.}, $M_v^H$) to indicate the hybrid system associated to a symbol. 

\begin{remark}
Undoubtedly, it would have been more mathematically elegant to work with the standard presentation of directed reflexive graphs (with a distinguished edge for each vertex), allowing us to combine the notions of flow set and guard set into a single abstract concept. However, this would be a significant notational departure from existing hybrid systems literature, including the references  \cite{simic, lerman, htp:bisimulations} given above as well as the increasingly influential approach described in \cite{gst:hybrid_systems}, where the considerations of how an execution under investigation crosses through flow sets (where continuous evolution is available) and guard (jump) sets (where discrete evolution is available) are of central importance.
\end{remark}

We begin with some simple examples to illustrate the definitions.

\begin{example}[Continuous and discrete systems]
Given any complete vector field $X$ on a manifold $M$, there is a hybrid system whose graph has a single vertex $v$ and no resets, where  $F_v = I_v = M_v = M$ and $X_v = X$.

Given any function $f\colon M \to M$ on a manifold $M$, there is a hybrid system whose graph has a single mode $v$ and single reset $e$, where $Z_e = I_v = M_v = M$, $F_v = \emptyset$,  and $r_e = f$. \qed
\end{example}

\begin{example}[Rocking block (\Cref{fig:rocking-block1}), following \cite{lygeros2003dynamical}]

\begin{linenomath*}
\begin{figure*}[t!]
    \centering
    \begin{subfigure}[t]{0.4\textwidth}
        \centering
        \includegraphics[height=2in]{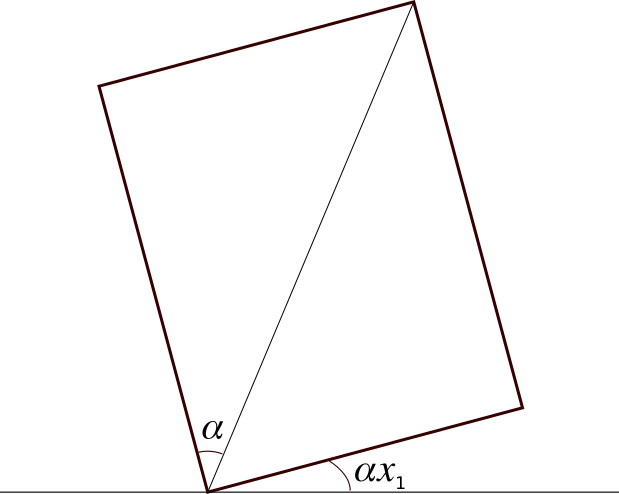}
        \caption{The block in continuous mode $L$ (left-leaning) }
    \end{subfigure}%
    \hfill
    \begin{subfigure}[t]{0.5\textwidth}
        \centering
        \includegraphics[height=2in]{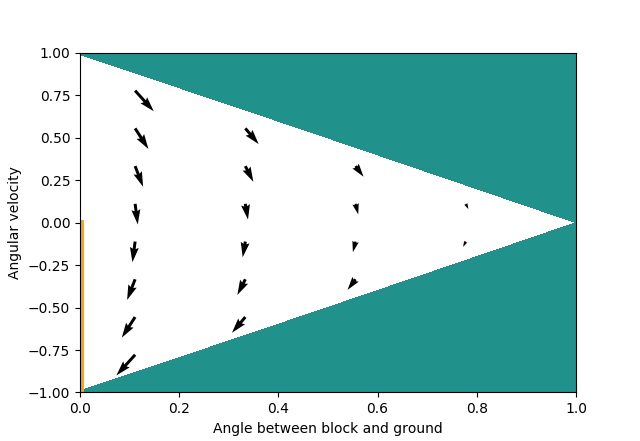}
        \caption{The continuous mode $(I_L, X_L)$ with guard set $Z_{LR}$ in orange}
    \end{subfigure}
\caption{A rocking block}
\label{fig:rocking-block1}
\end{figure*}
\end{linenomath*}

We can model a non-slipping, rocking rectangular block as a hybrid system $H$ with two continuous modes $V(H) = \{L, R\}$ corresponding to leaning to the left and right, respectively.  The system has two reset maps corresponding to the bottom of the block hitting the ground as the block switches pivots:
\begin{linenomath*}
\[
    E(H) = \{e_{LR} \colon L \to R, \quad e_{RL} \colon R \to L\}.
\]
\end{linenomath*}
Let $\alpha$ be the angle between either diagonal and the vertical face meeting the diagonal at the bottom of the block.  We assume that a constant fraction $r$ of the angular velocity is lost at each reset.

For each $v \in V(H)$, we let $M_v = \mathbb R^2$. We define the active set to be
\begin{linenomath*}
\[
    I_v = \left\{ (x_1, x_2) \in \mathbb{R}^2 \mid 0 \le x_1 \le 1 \text{ and }  \cos(\alpha (1 - x_1)) + \frac{(\alpha x_2)^2}{2} \le 1 \right\} 
\]
\end{linenomath*}
where $x_1$ is the absolute value of the angle between the horizontal block face and the ground (as a fraction of $\alpha$), and $x_2$ is the angular velocity (up to a sign).  
The guard sets are
\begin{linenomath*}
\[
    Z_e = \{ x \in I_{\src(e)} \mid x_1 = 0 \text{ and } x_2 \leq 0 \},
\]
\end{linenomath*}
and the flow sets are $F_v = I_v \setminus Z_v$.
The continuous dynamics are given by
\begin{linenomath*}
\[
    X_v = 
    \begin{pmatrix}
    x_2 \\
    -\alpha^{-1} \sin(\alpha (1-x_1))
    \end{pmatrix}
\]
\end{linenomath*}
for each vertex $v$.
The reset maps are
\begin{linenomath*}
\[
    r_e(0, x_2) = (0, -rx_2)
\]
\end{linenomath*}
for each $e \in E(H)$.\qed
\end{example}

\subsection{Smooth sets} \label{sec:smooth}
Since the dynamics of a hybrid system occur in the active sets, it makes sense to define hybrid semiconjugacies via maps of active sets, not their ambient manifolds. However, a typical active set will not even form a manifold with corners much less a manifold.  Nonetheless, it will be important for us to not abandon the smooth setting completely and work, for example, with arbitrary set functions which would not give a sufficient foundation for our constructions. Thus, we will define ``smooth maps'' of such sets, which although requiring a somewhat technical definition, will provide a solid basis and be used extensively in sequel. The reader not interested in the intricacies of these definitions is safe to skip this section and conceptually think of maps of smooth sets as a technical generalization of the usual smooth maps of manifolds. 

\begin{defn}
\label{def:smooth_set}
A \textbf{smooth set} $A$ is a pair $(A, M_A)$ where $M_A \in \mfld$ is a manifold and $A \subset M_A$ is an arbitrary subset of $M_A$.  We will often abuse notation by referring to a smooth set $(A, M_A)$ by the name of its underlying set $A$ with the understanding that there exists some fixed ambient manifold $M_A$ containing $A$.

Let $A$ and $B$ be smooth sets. Let $N \subset M_A$ be an open neighborhood of $A$ and $f \colon N \to M_{B}$ be a smooth map such that $f(A) \subset B$.  
We say that two such smooth maps $f\colon N \to M_B$ and  $f'\colon N' \to M_B$  are \textbf{germ-equivalent} on $A$ if there exists an open neighborhood $U$ of $A$ such that $U \subset N \cap N'$ and $f|_U = f'|_U$.  
Thus, $f$ is germ-equivalent to $f'$ if and only if $f$ and $f'$ define the same germ at every point in $A$. It is clear that this relation is reflexive, symmetric, and transitive. We will use the term \textbf{smooth map} $\alpha\colon A \to B$  to refer to an equivalence class of such maps.  Thus, $\alpha$ consists of a set map from ${A}$ to ${B}$ along with an germ-equivalence class of extensions to neighborhoods of $A$.

Let $A$, $B$, and $C$ be smooth sets. Suppose $\alpha\colon A \to B$ and $\beta\colon B \to C$ are smooth maps. Let $\widetilde \alpha\colon U_{ \alpha} \to M_{B}$ and $\widetilde \beta\colon U_{ \beta} \to M_{C}$ be representatives for $\alpha$ and $\beta$, respectively.  Since $\widetilde\alpha$ is continuous, the set $\widetilde{\alpha}^{-1}(U_\beta) \subset U_\alpha$ is an open neighborhood of $A$ in $M_{A}$.  Let $U_{{\beta \circ \alpha}} =  \widetilde\alpha^{-1}(U_\beta)$ and  $\widetilde{\beta \circ \alpha} = \widetilde{\beta} \circ \widetilde{\alpha}|_{\widetilde \alpha^{-1}(U_\beta)}$.  Since restriction of $U_\alpha$ and $U_\beta$ results in a restriction of $\widetilde{\beta \circ \alpha}$,  this defines a smooth map $\beta \circ \alpha $. Composition of the representatives is associative, so composition of smooth maps between smooth sets is also associative.  Lastly, if $A$ is a smooth set, then the equivalence class of the pair $(M_A, \id_{M_A})$ defines an identity map.  Thus, smooth sets and smooth maps form a {\bf category of smooth sets}, which will denote $\sset$.  

We will say that a smooth map between smooth sets is a {\bf diffeomorphism}, {\bf embedding}, or {\bf submersion} if any of its representatives is.  Similarly, if $(A, M_A)$ and $(B, M_{B})$ are smooth sets and $X\colon A \to TM_A$ and $Y\colon B \to TM_B$ are vector fields, we will say that a smooth map $\alpha \colon A \to B$ is a {\bf smooth semiconjugacy} if any of its representatives $\widetilde \alpha\colon (U_{\alpha}, X|_{U_\alpha}) \to (M_{B}, Y)$ satisfies $Y \circ \widetilde \alpha = (T \widetilde \alpha) \circ X$, where $T\widetilde\alpha$ is the differential of $\alpha$. 
\end{defn}

We note that the definition of a smooth map of subsets of manifolds is not universally agreed upon. An alternative definition ({\em e.g.}, in \cite{lee:smooth_manifolds}) requires the mere existence of a smooth extension: if $M, N \in \mfld$ and $Z$ is an arbitrary subset of $M$, then a map $\alpha\colon Z \to N$ is smooth in this sense if for every point $x \in Z$, there is some neighborhood $W_x$ of $x$ and a smooth map $\widehat{\alpha}_x\colon W_x \to N$ with $\widehat{\alpha}_x|_{Z \cap W_x} = \alpha|_{Z \cap W_x}$. Using a partition of unity argument, we can see that this is equivalent to requiring that there is a neighborhood $W$ of $Z$ in $M$ and a smooth map $\widehat\alpha\colon W \to N$ that extends $\alpha$.

However, the category arising from such a notion of smooth map is not well-behaved with respect to local conditions.  For example, an isomorphism, {\em i.e.}, a bijective map $f\colon Z \to Z'$ such that both $f$ and $f^{-1}$ have smooth extensions, need not extend to a local diffeomorphism in the ambient manifold.  Thus, any local condition on a smooth map ({\em e.g.}, being a submersion) may not preserved by composition with isomorphisms in this category. Our category of smooth sets rectifies this problem.  Using the invariance of domain theorem, one can show that the isomorphisms in $\sset$ are precisely the diffeomorphisms of smooth sets.   

As a consequence, since diffeomorphisms are defined locally the ambient manifold $M_A$ of a smooth set $A$ is largely irrelevant away from a small neighborhood of $A$.  For instance, given an embedding of smooth sets $i\colon A \hookrightarrow B$, there might not be a global embedding $M_A \hookrightarrow M_B$, but there will be a diffeomorphism of smooth sets $({A}, M_A) \to ({A}, U_i)$ realizing the restriction from $M_A$ to some neighborhood $U_i$ of $A$ and a smooth embedding $U_i \hookrightarrow M_B$ in $\mfld$. 

There is a well-defined notion of restriction for smooth maps as well.   If $f\colon A \to B$ is a smooth map with representative $(U_f, \tilde f)$ and $C \subset A$ is an arbitrary subset, then the {\bf restriction} $f|_C\colon C \to B$ is the smooth map with representative $(U_f, \tilde f)$. Equivalently, $f|_C = f \circ i_C$ where $i_C\colon C \to A$ is the smooth map represented by $\id_{M_A}$.  

We must exercise caution, however, in considering subobjects in $\sset$. For example, if $i\colon A \to B$ is an embedding of smooth sets that is not a local diffeomorphism ({\em i.e.}, $i$ is not a submersion), then $A \not\simeq i(A)$ as smooth sets.  For example, a point embedded in $\mathbb R$  is not isomorphic to a point embedded in $\mathbb R^2$.

\begin{remark}
We note that the construction described in \Cref{def:smooth_set} for defining smooth maps arises naturally in many contexts. More abstractly, we could say that, if $A$ and $B$ are smooth sets,  
\begin{linenomath*}
\[
    \Hom_{\sset}(A, B) = \colim_{U \supseteq A} \{f \in \Hom_{\mfld}(U, M_A) \mid f(A) \subset B\},
\]
\end{linenomath*}
where the colimit is over the partially ordered collection of open sets $U$ containing $Y$. In sheaf-theoretic settings, this filtered colimit is often called a ``direct limit'' and the  construction here is equivalent to taking a subset of the global sections of the presheaf $i^{-1}\Hom(M_A, M_B)$ for the inclusion $i\colon A \hookrightarrow M_A$, where $\Hom(M_A, M_B)$ is considered as a sheaf on $M_B$ in the usual way. 
\end{remark}

\subsection{Hybrid semiconjugacies}

We turn our attention to morphisms of hybrid systems. The following definition is an adaption of Lerman's definition of hybrid semiconjugacy \cite{lerman}.  The primary difference is that we allow more general directed graph morphisms in $\dgraph$ as described in \Cref{sec:graphs} above. The intuition for this definition is an extension of the idea of a smooth semiconjugacy: the dynamics of the domain should be consistent with the dynamics of the image. Specifically, condition $(1)$ ensures that the continuous dynamics are compatible, while condition $(2)$ requires that the discrete dynamics are also consistent.

\begin{defn} \label{def:hyb_sem}
A {\bf hybrid semiconjugacy} $\alpha \colon H \to K$ of hybrid systems consists of a graph morphism $G(\alpha) \colon \graph G(H) \to \graph G(K)$ (which will also refer to as $\alpha$, abusing notation) and a collection of smooth maps $\alpha_v \colon I_v^H\to I^K_{\alpha(v)}$ (in the sense of \Cref{def:smooth_set}) for $v \in V(H)$ such that
\begin{enumerate}
    \item for each $v \in V(H)$,  we have $\alpha_v(F^H_v) \subset F^K_{\alpha(v)}$ .  Moreover, the restriction $\alpha_v|_{F^H_v}\colon F^H_v \to F^K_{\alpha(v)}$ is a smooth semiconjugacy;
    \item For each reset $e \in E(H)$, we have $\alpha_{\src(e)}(Z^H_e) \subset Z^K_{\alpha(e)}$, where if $\alpha(e) \in V(K)$ then $Z^K_{\alpha(e)} := F^K_{\alpha(e)}$. Moreover, the square

\begin{linenomath*}
\begin{equation}
    \label{eq::lerman_2}
    \begin{tikzpicture}[baseline=(current bounding  box.center)]
        \node (x) at (0,0) {$Z^H_e$};
        \node (y) at (4,0) {$I_{\tgt(e)}^H$};
        \node (z) at (0,-2) {$Z^K_{\alpha(e)}$};  
        \node (w) at (4, -2) {$I_{\tgt(\alpha(e))}^K$};
        
        \draw[->, above] (x) to node {$r_e^H$} (y);
        \draw[->, left] (x) to node {$\alpha_{\src(e)}|_{Z_e^H}$} (z);
        \draw[->, right] (y) to node {$\alpha_{\tgt(e)}$} (w);
        \draw[->, below] (z) to node {$r_{\alpha(e)}^K$} (w);
    \end{tikzpicture}
    \end{equation}
\end{linenomath*}
commutes, where if $\alpha(e) \in V(\graph K)$, then the map $r^K_{\alpha(e)}$ in $(\ref{eq::lerman_2})$ should be interpreted as the inclusion map $F^K_{\alpha(e)} \hookrightarrow I^K_{\alpha(e)}$.
\end{enumerate}

 Since the squares above can be composed vertically, we can then define a category $\hempty$ of hybrid systems and hybrid semiconjugacies.  To avoid proliferation of empty  active sets when taking fiber products, we will primarily work with the full subcategory $\hs$ of hybrid systems which have no empty guard sets or active sets. We will clearly indicate the few occasions when it is necessary for us to work in the larger category $\hempty$.
\end{defn}

Since hybrid semiconjugacies are collections of maps of smooth sets, it is better to think of a continuous mode of a hybrid system as a smooth set $(I_v, M_v)$ with extra structure, rather than a manifold $M_v$ with extra structure.  In fact, the ambient manifold $M_v$ is irrelevant away from $I_v$.  More precisely, given a hybrid system $H \in \hs$ and any open neighborhood $U_v \subset M_v$ of an active set $I_v$, there is an isomorphism in $\hs$ replacing $M_v$ with $U_v$.

\begin{example}
\label{ex:initial_terminal_object}
Let $\one$ be the hybrid system with a single continuous mode where $M^{\one} = I^{\one} = \{\ast\} = F^{\one}$ with the zero vector field and no resets. Then $\one$ is terminal in $\hs$ with the unique morphism defined by constant maps.  

There is also an initial system in $\hs$ given by the hybrid system with empty graph and correspondingly empty collections of continuous modes and reset maps. \qed
\end{example}

We now define several distinguished classes of morphisms that will play an important role in particular for the investigation of template-anchor pairs in \Cref{sec:templates_anchors}.

\begin{defn}
An $\hs$-morphism $f\colon H \hookrightarrow K$ is a {\bf hybrid embedding} if the corresponding graph morphism $f\colon \graph G(H) \to \graph G(K)$ is monic in $\dgraph$ and for each $v \in V(H)$, the smooth map $f_v$ is an embedding. In this case, we will say that $H$ is a {\bf subsystem} of $K$ and write $H \subset K$, identifying $H$ with its image under $f$. Note that the composition of hybrid embeddings is also a hybrid embedding so that we get a subcategory $\hsem$ of $\hs$ whose morphisms are the hybrid embeddings.  
\end{defn}

\begin{defn}
A {\bf hybrid submersion} is an $\hs$-morphism $p\colon H\to K$ such that
\begin{enumerate}
    \item[(i)] for each $v \in V(H)$, $p_v$ is a submersion of smooth sets.
\end{enumerate}
If additionally, $p$ satisfies 
\begin{enumerate}
    \item[(ii)] $p\colon \graph G(H) \to \graph G(K)$ is an epimorphism in $\dgraph$;
    \item[(iii)] for each $u \in V(K)$, we have 
    \begin{linenomath*}
    \[
        I_u = \bigcup_{v \in p^{-1}(u)} p_v(I_v),
    \]
    \end{linenomath*}
\end{enumerate}
then we will say that $p$ is a {\bf hybrid surjective submersion}. 
\end{defn}

 It is straightforward to check that hybrid embeddings and hybrid surjective submersions are monic and epic in $\hs$, respectively. It would be desirable to completely characterize hybrid embeddings and hybrid surjective submersions via universal properties, potentially simplifying the proofs below. However, such a characterization seems unlikely since, even in the category of smooth manifolds, monomorphisms and epimorphisms are relatively wild (injective smooth maps and smooth maps with dense image, respectively).

\subsection{Hybrid state spaces}
In the previous sections, we defined a hybrid system via data associated to an underlying graph.  An alternative perspective is to consider a hybrid system as data associated to an underlying hybrid state space.

\begin{defn}
\label{def:hybrid_state_spaces}
We define the {\bf hybrid state space} of a hybrid system $H$ to be the topological space 
\begin{linenomath*}
\[
I(H) := \bigsqcup_{v \in V(H)} I_v.
\] 
\end{linenomath*}
Since a morphism $f\colon H \to K$ in $\hs$ gives a continuous map $I(H) \to I(K)$, we get a functor $I\colon \hs \to \tops$.

We will sometimes suppress mentioning the $I$ functor to simplify notation.  For example, if $f\colon H \to K$ is any $\hs$-morphism and $B \subset I(K)$, it will sometimes be convenient for us to write $f^{-1}(B)$ for the topological preimage $I(f)^{-1}(B)$ in $I(H)$.  Similarly, we will sometimes write $f(H)$ for the topological image $I(f)(I(H))$.  Given $W \subset I(H)$, we will use $W^\circ$ and $\overline W$ for the interior and closure of $W$ in $I(H)$, respectively.
\end{defn}

\begin{remark}
 In \cite{lerman}, a mathematically elegant definition of hybrid state space is given that incorporates the guards and reset maps (in that setting, reset maps are also allowed to be more general relations) into the definition of the hybrid state space. Our formulation here reflects our design decision to consider both the continuous dynamics (vector fields) and discrete dynamics (guards and reset maps) as being additional data to be specified on top of an underlying hybrid state space. In some cases of particular interest for our targeted applications this is natural, as there is a fixed ambient state space (often a subset of $\mathbb{R}^n$) modeling the physical states of the system, while various combinations of  vector fields, guards, and reset maps are used as controls. 
\end{remark}

\subsection{Executions}
We use a notion of hybrid execution similar to \cite{lygeros:executions}, but reformulated in terms of our hybrid semiconjugacies following \cite{lerman, htp:bisimulations}.   Just as an integral curve for a smooth system $(M,X)$ is a semiconjugacy $(J, \frac{d}{dt}) \to (M,X)$ for some interval $J$, we will define a hybrid execution to be a semiconjugacy $\tau \to H$ where $\tau$ is a hybrid time trajectory, the hybrid analog of an interval.  An example hybrid time trajectory is depicted in \Cref{fig:hybrid_time_trajectory}.

\begin{defn} \label{def:htt}
A {\bf hybrid time trajectory} $\tau$ is a pair $((\tau_j)_{j=0}^N, c_\tau)$  where
\begin{itemize}
\item  $(\tau_j)_{j=0}^N \subset [0, \infty]$ is a nondecreasing sequence such that $1 \le N \le \infty$, $\tau_0 = 0$, and $\tau_j \neq \infty$ for $j < N$; and
\item  the {\bf final endpoint type} $c_\tau \in \{\text{open}, \text{closed}\}$ can only be closed if $N < \infty$.  Moreover, if $N < \infty$ and $\tau_N = \tau_{N-1}$, then $c_\tau$ must be closed.
\end{itemize}

To any hybrid time trajectory $\tau$, we associate the following hybrid system:
\begin{itemize}
    \item Let $G(\tau)$ be a directed path with vertices $(v_j)_{j=0}^{N-1}$ and edges $(e_j)_{j=0}^{N-2}$
    \item For all $0 \le j < N$, let $(M_{v_j}, X_{v_j}) = (\mathbb{R}, \frac{d}{dt})$ and $F_{v_j} = [\tau_j, \tau_{j+1})$. In particular, $F_{v_j} = \emptyset$ if $\tau_j = \tau_{j+1}$.
    \item Let
    \begin{linenomath*}
    \[
    I_{v_j} = 
    \begin{cases}
    [\tau_{j}, \tau_j) &  j = N-1 \text{ and } c_\tau = \text{open} \\
    [\tau_{j}, \tau_{j+1}] &   \text{otherwise}
    \end{cases}
    \]
    \end{linenomath*}
    \item  For $j < N-1$, let $Z_{e_j} = \{\tau_{j+1}\}$, and $r_{e_j}\colon Z_{e_j} \to I_{v_{j+1}}$ be the map induced by $\id_{\mathbb{R}}$. 
\end{itemize} 

If $1 \le j < N$, then we say that $\tau_j$ is a {\bf jump time} of $\tau$. We define the {\bf stop time} of $\tau$ to be
\begin{linenomath*}
\[
    \tau^\stp :=
    \begin{cases}
    \tau_N & N < \infty\\
    \displaystyle \lim_{j \to \infty} \tau_j & N = \infty
    \end{cases}
\]
\end{linenomath*}

If $\tau^\stp = \infty$, we say that $\tau$ is {\bf infinite}. If $N < \infty$ and $\tau^\stp < \infty$, we say that $\tau$ is {\bf finite}. Finally, if $N = \infty$, but $\tau^\stp < \infty$, we say that $\tau$ is  {\bf Zeno} with Zeno time $\tau^\stp$. 

For each $t \geq 0$, we define the set of {\bf time $\bm{t}$-points} of $\tau$ to be 
\begin{linenomath*}
\[
    \tau(t) = \{x \in I(\tau) \mid x = t \}.
\]
\end{linenomath*}
In particular, we say that $\min(I_{v_0})$ is the {\bf starting point} of $\tau$.   If $N < \infty$, $\tau_N < \infty$, and $c_\tau = \text{closed}$, we say that $\tau$ has an {\bf end point}, namely $ \max(I_{v_{N-1}})$.
\end{defn}

\begin{defn}
Let $\tau$ and $\sigma$ be hybrid time trajectories. We say that $\tau$ is a {\bf prefix} of $\sigma$ if there exists a hybrid embedding $i\colon \tau \to \sigma$ mapping the starting point of $\tau$ to the starting point of $\sigma$, i.e. $i(v_0^\tau) = v_0^\sigma$ and $i_{v_0^\tau}(0) = 0$.
\end{defn}

\begin{linenomath*}
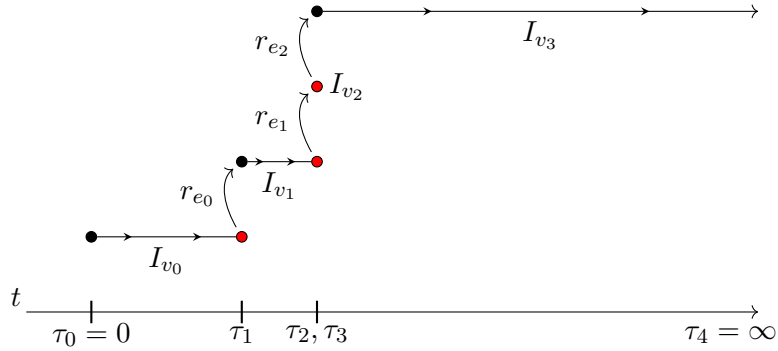
\begin{figure}
\begin{tikzpicture}
    \node[endpoint] (t0) at (0,0) {};
    \node[guard] (t1) at (2,0) {};
    \node[endpoint] (t2) at (2,1) {};
    \node[guard] (t3) at (3,1) {};
    \node[guard] (t4) at (3,2) {};
    \node[endpoint] (t5) at (3,3) {};
    \node (t6) at (9,3) {};
    \node (tbeg) at (-1,-1)  {};
    \node (tend) at (9,-1) {};
    
    \node (zero) at (0,-1.3) {$\tau_0 = 0$};
    \draw[thick] (0, -.85) -- (0, -1.15);
    \node (tau1) at (2,-1.3) {$\tau_1$};
    \draw[thick] (2, -.85) -- (2, -1.15);
    \node (tau23) at (3,-1.3) {$\tau_2,\tau_3$};
    \draw[thick] (3, -.85) -- (3, -1.15);
    \node (tau4) at (8.5,-1.3) {$\tau_4 = \infty$};
    \node (v2) at (3.4, 2) {$I_{v_2}$};

    \node (t) at (-1, -.8) {$t$};
    
    \draw[-, postaction={mid arrows},below] (t0) to node {$I_{v_0}$} (t1);
    \draw[-, postaction={mid arrows},below] (t2) to node {$I_{v_1}$} (t3);
    \draw[->, postaction={mid arrows},below] (t5) to node {$I_{v_3}$} (t6);
    
    \draw[->] (tbeg) to (tend);
    
    \draw[->, shorten <=2pt, shorten >=2pt, left, out=120, in=220] (t1) to node {$r_{e_0}$}(t2);
    \draw[->, shorten <=2pt, shorten >=2pt, left, out=120, in=220] (t3) to node {$r_{e_1}$}(t4);
    \draw[->, shorten <=2pt, shorten >=2pt, left, out=120, in=220] (t4) to node {$r_{e_2}$}(t5);
\end{tikzpicture}
\caption{A pictorial representation of a hybrid time trajectory with four continuous modes aligned vertically with the corresponding points on the non-negative real line below. The guards are marked in red.}
\label{fig:hybrid_time_trajectory}
\end{figure}
\end{linenomath*}

\begin{defn} \label{def:execution}
An {\bf execution} of a hybrid system $H$ is a hybrid semiconjugacy $\chi\colon \tau \to H$ for some hybrid time trajectory $\tau$. 

We call the set of edges $e \in E(\tau)$ such that $\chi(e) \in V(H)$ the {\bf $\chi$-trivial resets}.  If $\chi(e) \in E(H)$ for all $e \in E(\tau)$ ({\em i.e.} the execution has no trivial resets), we say that $\chi$ is a {\bf fundamental  execution}.  Every execution has an {\bf associated fundamental execution} $\chi^* \colon \tau^* \to H$  given by deleting trivial resets.\footnote{More precisely, for every (possibly infinite) maximal contiguous string of $\chi$-trivial resets $T = (e_j)_{j=k}^m \in E(\tau)$,
(i) replace the edge-induced subgraph of $T$ in $G(\tau)$ with a single vertex $v$ , (ii) 
set $I_v := I_{v_k} \cup I_{v_{m+1}}$ if $m$ is finite and $I_v := I_{v_k}$ otherwise, and (iii) define $\chi_v^*$ by 
($\chi^*_v)_{|I_{v_k}} = \chi_{v_k}$ and, if $m$ is finite, $(\chi^*_v)_{|I_{v_{m+1}}} = \chi_{v_{m+1}}$.}

We say that $\chi$ is {\bf infinite}, {\bf finite}, or {\bf Zeno} if $\tau$ has the corresponding property. 
We say that $\chi$ is a {\bf prefix} of an execution $\psi\colon \upsilon \to H$ if $\tau$ is a prefix of $\upsilon$ and $\chi = \psi \circ i$ where $i\colon \tau \to \upsilon$ is the canonical embedding. An execution is {\bf maximal} if it is not the prefix of another execution; in particular, all infinite executions are maximal. We let $\exc_H$ denote the set of executions of $H$, and let $\infexc_H \subset \exc_H$ denote the infinite executions.  We let $\exc_H(x)$ denote the set of executions starting at a point $x \in I(H)$, and similarly for $\infexc_H(x)$.
\end{defn}

\begin{remark}
\label{rem:zeno_continuations}
Notice that this definition of execution allows for infinitely many jumps in finite time, but does not allow for subsequent execution after the stop time. In particular, we allow Zeno executions as described above, but do not permit explicit continuations of Zeno executions to be considered as a single execution. 
\end{remark}

\subsection{Deterministic and nonblocking systems}
We now define deterministic and nonblocking hybrid systems, two well-behaved subclasses of hybrid systems whose intersection is analogous to the subclass of complete vector fields within all vector fields.

\begin{defn} \label{def:deterministic}
A hybrid system $H$ is {\bf deterministic} if for every $v \in V(H)$, (i) the collection of sets $\{F_v\} \cup \{ Z_e \mid \src(e) = v \}$ is pairwise disjoint and (ii) for every point $x \in F_v$ the set of $X_v$-integral curves $J \colon [0,T) \to F_v$ with $J(0) = x$ has a unique maximal element with respect to the domain-restriction relation.
\end{defn}

We justify this definition with the following simple proposition (cf. \cite[Lemma~III.2]{lygeros2003dynamical}).

\begin{prop} \label{prop:det}
Let $H$ be a deterministic system and $x \in I(H)$.  Then $\exc_H(x)$ contains a unique maximal fundamental execution.
\end{prop}
\begin{proof}
Let $\chi, \chi' \in \exc_H(x)$ be maximal fundamental executions. Let $(\tau_n)$ and $(\tau_n')$ be the jump times of $\chi$ and $\chi'$, respectively, and similarly for the corresponding intervals $I_{v_n}$ and $I_{v_n}'$.  We will argue by induction on $n$ that $I_{v_n} = I_{v_n}'$ and $\chi_{|I_{v_n}} = \chi_{|I_{v_n}}'$ for all $n$, which implies the desired result. 

For $n = 0$, we have $\tau_0 = 0 = \tau_0'$.  If $x \in Z_e$ for some edge $e \in E(H)$, then $\tau_1 = 0 = \tau_1'$, so $I_{v_0} = I_{v_0}'$ and $\chi_{|I_{v_0}} = \chi_{|I_{v_0}}'$.    If $x \in F_v$ for some vertex $v \in V(H)$, then  \Cref{def:deterministic} implies that (i) $\tau_1 = \tau_1'$ so $I_{v_0} = I_{v_0}'$ (ii) $\chi$ and $\chi'$ agree on $[0,\tau_1)$ and hence on $[0,\tau_1]$ (if $I_{v_0}$ is right-closed) by continuity. 

For the inductive step, suppose that $I_{v_n} = I_{v_n}'$ and $\chi_{|I_{v_n}} = \chi_{|I_{v_n}}'$.  If $\chi_{|I_{v_n}}(\tau_{n+1})$ either (i) does not exist or (ii) does not lie in a guard set, then we are done.  Otherwise, we have $\chi_{|I_{v_n}}(\tau_{n+1})  = \chi_{|I_{v_n}}'(\tau_{n+1}') \in Z_e$ for some edge $e \in E(H)$.  Thus, $\tau_{n+1} = \tau_{n+1}'$ and $\chi_{I_{v_{n+1}}}(\tau_{n+1}) = r_e(\chi_{I_{v_n}}(\tau_{n+1})) = \chi'_{I_{v_{n+1}}}(\tau_{n+1})$.  By the same argument as in the preceding paragraph, it follows that $I_{v_{n+1}} = I_{v_{n+1}}'$ and $\chi_{|I_{v_{n+1}}} = \chi_{|I_{v_{n+1}}}'$.
\end{proof}

The usual definition of determinism (uniqueness of maximal executions) corresponds to the conclusion of \Cref{prop:det}. The main advantage of using \Cref{def:deterministic} instead is that products of deterministic systems remain deterministic (\Cref{cor:dn_prod}).  Simple examples demonstrate this result fails to hold in the absence of the disjointness condition. 

Conversely, if a hybrid system $H$ satisfies the conclusion of \Cref{prop:det}, then the following slight modification of $H$ is deterministic in the sense of \Cref{def:deterministic}. We first note that distinct guard sets of $H$ cannot overlap because any execution starting at a point of overlap could immediately reset via either associated reset. On the other hand, there \emph{can} be overlap of guard sets and flow sets ({\em e.g.}, $(F_v, X_v) = ([0,1], \frac{d}{dt})$ with $Z_v = \{1\}$), but removing the points of overlap from the flow sets results in a hybrid system with the same active sets and maximal executions that is deterministic in the sense of \Cref{def:deterministic}.

\begin{defn} \label{def:nonblocking}
We say that a hybrid system $H$ is {\bf nonblocking} if for all $x \in I(H)$ there exists an infinite execution starting at $x$. 
\end{defn}

We denote the full subcategory of $\hs$ of deterministic nonblocking systems by $\hsdn$. 

\begin{example}
Every hybrid time trajectory $\tau$ is deterministic.  If $\tau$ is infinite, then $\tau$ is also nonblocking.  \qed
\end{example}

\section{Categorical properties of \texorpdfstring{$\hs$}{\bf H}} \label{sec:cat}

The following basic results provide a basis for the compositional constructions below.  In particular,  we show that the category of hybrid systems $\hs$ is fibered over  $\dgraph$, our alternative presentation of the category of reflexive graphs.  Applying this result, we construct fiber products of hybrid semiconjugacies along hybrid submersions.   As a corollary, we find that $\hs$ and $\hsdn$ are cartesian, thus supporting the most basic form of parallel composition. 

To study limits in $\hs$, we begin with the corresponding notion for graphs. The following well-known lemma gives this construction. The proof of the lemma is straightforward, but we provide it here for convenience.
\begin{lemma}
\label{lem:digraph_complete}
The category $\dgraph$ is complete and cocomplete. 
\end{lemma}
\begin{proof}
As usual, it is enough to see that $\dgraph$ has all (co)products and also all (co)equalizers. Given a collection of directed graphs $\{G_j\}_{j \in J}$ indexed by a set $J$, let $V\left(\prod_{j \in J} G_j\right) = \prod_{j \in J} V(G_j)$ and $V(\coprod_{j \in J} G_j) = \bigsqcup_{j \in J} V(G_j)$. If we set $P_j = E(G_j) \sqcup V(G_j)$ as generalized edges, then let 
\begin{linenomath*}
\begin{gather*}
    E\left(\prod_{j \in J} G_j\right) = \prod_{j \in J} P_j \setminus \prod_{j \in J} V(G_j)\\
    E\left(\coprod_{j \in J} G_j \right) = \bigsqcup_{j \in J} E(G_j)
\end{gather*}
\end{linenomath*}
with the obvious source and target maps for each. Then we have projections $\pi_i\colon \prod G_j \to G_i$ and canonical injections $k_i\colon G_i \to \coprod G_j$ which are just the set-theoretic projections and injections. It is straightforward to check that these constructions are universal. 

For (co)equalizers of graph morphisms $f, g\colon G \to H$, we can simply take the (co)equalizers of the set maps $f_V, g_V$ and $f_E, g_E$.
\end{proof}    

We  denote the product of two graphs $G$ and $H$ by $G \boxtimes H$ to emphasize the difference between this product and what is sometimes referred to as the ``categorical product'' of graphs. That product, where the edge set for the product is simply the product of the edge sets, is the natural product in the category where edges are necessarily mapped only to edges.

To extend this construction to the hybrid setting, we prove that the following functors are Grothendieck fibrations (see \cite{loregian-riehl:fibrations}, Section 2.2 for excellent recent exposition on categorical fibrations).  Roughly speaking, this means that we can form ``pullbacks'' of hybrid systems along graph morphisms analogous to pullbacks of vector fields along smooth maps.

\begin{prop}
\label{prop:fibration}
The forgetful functors $G\colon \hempty \to \dgraph$ and $G_{|H}\colon \hs \to \dgraph$ are fibrations.
\end{prop}
\begin{proof}
Let $\mathscr C$ be the class of hybrid semiconjugacies $H \xrightarrow{f} K$ in $\hempty$ (respectively, $\hs$) such that for each $v \in V(H)$ the map $f_v\colon I_v^H \to I_{f(v)}^K$ is a diffeomorphism. 
We claim that each $f \in \mathscr C$ is a cartesian morphism. To see this, suppose $g\colon H' \to K$ is another hybrid semiconjugacy and that $\varphi\colon G(H') \to G(H)$ is a graph morphism $G(f) \circ \varphi = G(g)$. We claim that  the graph morphism $\varphi$ lifts to a hybrid semiconjugacy $ \tilde\varphi \colon H' \to H$ by letting $\tilde\varphi_v(x) = f_{\varphi(v)}^{-1} \circ g_v$, which is well defined since any smooth map representative of $f_{\varphi(v)}$ is a diffeomorphism onto a neighborhood of $I_{g(v)}^K$. 

The fact that $\tilde\varphi$ is a hybrid semiconjugacy then follows from the commutativity of the reset diagrams for $g$ and $f$.  Moreover, the lift $\tilde\varphi$ is unique, since any other such lift must have the same underlying continuous semiconjugacies. 

To see that $G$ is a fibration, let $H$ be a hybrid system and $\varphi\colon \Gamma \to G(H)$ a graph morphism for some graph $\Gamma$. Then we can obtain a hybrid system $K_\varphi$ with underlying graph $\Gamma$, pulling back the dynamics of $H$ along $\varphi$. 

More precisely,  we can lift $\varphi$ to a cartesian hybrid semiconjugacy $\tilde\varphi\colon K_\varphi \to H$ by setting 
$G(K_\varphi) = \Gamma$. Then for each $v \in V(K_\varphi)$ let $(I^{K_\varphi}_v, F^{K_\varphi}_v, X^{K_\varphi}_v) = (I^H_{\varphi(v)}, F^H_{\varphi(v)}, X^H_{\varphi(v)})$ and for each $e \in E(K_{\varphi})$ let $Z^{K_\varphi}_e = Z^H_{\varphi(e)}$ when $\varphi(e) \in E(H)$ or $Z^{K_\varphi}_e = F^H_{\varphi(e)}$ if $\varphi(e) \in V(H)$. 
\end{proof}

Recall that the fiber categories for $G$ are a family of subcategories $\hs_{\Gamma}$ of $\hs$ indexed by objects of $\dgraph$ ({\em i.e.} each having a fixed graph $\Gamma$), where $\hs_\Gamma$ has objects hybrid systems $H \in \hs$ with $G(H) = \Gamma$ and morphisms $f\colon H \to K$ with $G(f) = \id_\Gamma$.

The following construction is well-known and generalizes in the obvious way to limits of diagrams of any shape provided that these limits exist in the base of a fibration and in the relevant fiber categories. We record here the proof for fiber products to set notation (for use in \Cref{prop:hybrid_fiber_products}, in particular).

\begin{lemma}
\label{lem:fiber_products_fibration}
If $F\colon\mathbf{A}\to\mathbf{B}$ is a fibration and both $\mathbf{B}$ and the fiber categories $\mathbf{A}_L$ have fiber products for $L \in \mathbf{B}$, then $\mathbf{A}$ has fiber products.
\end{lemma}
\begin{proof}
Suppose we have $A \xrightarrow{f} B \xleftarrow{g} C$ in $\mathbf{A}$. Then we can apply $F$ and obtain the fiber product in $\mathbf{B}$:
\begin{linenomath*}
\[
    \begin{tikzpicture}
        \node (x) at (0,0) {$L$};
        \node (y) at (3,0) {$F(A)$};
        \node (z) at (0,-2) {$F(C)$};
        \node (w) at (3, -2) {$F(B)$.};
        
        \draw +(.25,-.75) -- +(.75,-.75)  -- +(.75,-.25);

        \draw[->, above] (x) to node {$\varphi_A$} (y);
        \draw[->, left] (x) to node {$\varphi_C$} (z);
        \draw[->, right] (y) to node {$F(f)$} (w);
        \draw[->, below] (z) to node {$F(g)$} (w);
    \end{tikzpicture}
\]
\end{linenomath*}
Now we can use the fibration to obtain a cartesian $\mathbf{A}$-morphism $\tilde\varphi_A\colon L_A \to A$ lifting $\varphi_A$. Similarly, we can lift $\varphi_C$ and $\varphi_B := F(f) \circ \varphi_A$ to get the diagram in $\mathbf{A}$
\begin{linenomath*}
\[
    \begin{tikzpicture}
        \node (lt) at (0,0) {$\tilde L$};
        \node (la) at (3,0) {$L_A$};
        \node (lc) at (0,-2) {$L_C$};
        \node (lb) at (3, -2) {$L_B$};
        
        \node (c) at (0, -4) {$C$};
        \node (a) at (6,0) {$A$};
        \node (b) at (6, -4) {$B$};
        
        \draw +(.25,-.75) -- +(.75,-.75)  -- +(.75,-.25);

        \draw[->, above] (la) to node {$\tilde\varphi_A$} (a);
        \draw[->, left] (lc) to node {$\tilde\varphi_C$} (c);
        \draw[->, left] (lb) to node {$\tilde\varphi_B$} (b);
        \draw[->, right] (a) to node {$f$} (b);
        \draw[->, below] (c) to node {$g$} (b);
        
        \draw[->, below] (lc) to node {$\psi_g$} (lb);
        \draw[->, right] (la) to node {$\psi_f$} (lb);
        \draw[->] (lt) to (la);
        \draw[->] (lt) to (lc);
    \end{tikzpicture}
\]
\end{linenomath*}
where $\psi_f, \psi_g$ are the unique lifts of the identity $\id_L$ with respect to the cartesian morphism $\tilde\varphi_B$ and the compositions $f \circ \tilde\varphi_A, g \circ \tilde\varphi_B$, respectively. Then $\tilde L$ is the fiber product of $L_C \xrightarrow{\psi_C} L_B \xleftarrow{\psi_A} L_A$ in the fiber category $\mathbf{A}_L$. The fact that $\tilde L$ is also universal in $\mathbf{A}$ then follows from recalling that $F(\tilde{L})$ is universal in $\mathbf{B}$ and that $\tilde\varphi_C$ and $\tilde\varphi_A$ are cartesian in $\mathbf{A}$.  \end{proof}

Our last ingredient is the following simple lemma allowing us to translate limits in $\hempty$ into limits in $\hs$.  Constructing limits in the former category is more straightforward, but the latter category is more useful for applications. First, we require a definition.

\begin{defn}
Let $i\colon \hs \to \hempty$ denote the inclusion functor of the full subcategory, and $P \colon \hempty \to \hs$ be the {\bf pruning functor} defined on objets by deleting each continuous mode $v$ such that $I_v = \emptyset$ and each reset $e$ such that $Z_e = \emptyset$, and on morphisms by removing the smooth maps with empty domain. It is straightforward to check that $P$ is indeed functorial, using the conditions imposed in \Cref{def:hyb_sem} to ensure that $P$ is well-defined on morphisms.
\end{defn}

\begin{lemma} \label{lem:hempty}
The inclusion $i\colon \hs \to \hempty$ is left adjoint to $P\colon \hempty \to \hs$.   
\end{lemma}
\begin{proof}
For any hybrid system $H \in \hs$, the component of the unit $\eta_H\colon H \to P(i(H))$ is the identity $\id_H$ and for any $K \in \hempty$, the component of the counit $\varepsilon_K\colon i(P(K)) \to K$ is the semiconjugacy determined by the graph embedding $G(P(K)) \hookrightarrow G(K)$ with identity maps at each vertex.   It is easy to check that the counit-unit equations hold:
$\varepsilon_{i(H)} \circ i(\eta_H) = \id_{i(H)} \circ \id_{i(H)} = \id_{i(H)}$ for $H \in \hs$, and $P(\epsilon_K) \circ \eta_{P(K)} = \id_{P(K)} \circ \id_{P(K)} = \id_{P(K)}$ for $K \in \hempty$.
\end{proof}

\begin{cor}
If $D\colon \catname J \to \hs$ is a diagram in $\hs$ and $L \in \hempty$ is the limit of the diagram $i \circ D$, then $P(L)$ is the limit of $D$.
\end{cor}
\begin{proof}
Since $P$ is a right adjoint, it preserves limits (see Theorem 13.3.7 of \cite{barr-wells} or Theorem 4.5.2 of \cite{riehl:ctc}), so $P(L)$ is the limit of the diagram $P \circ i \circ D = D$. 
\end{proof}
Putting the preceding results together gives the following proposition, which generalizes the construction of fiber products of smooth maps along submersions to the hybrid setting.

\begin{prop}
\label{prop:hybrid_fiber_products}
If $p\colon K_1 \to H$ is a hybrid submersion and $f\colon K_2 \to H$ is any hybrid semiconjugacy, then the fiber product of $K_1 \xrightarrow{p} H \xleftarrow{f} K_2$ exists in $\hs$. 
\end{prop}
\begin{proof}
By \Cref{lem:hempty}, it suffices to define the fiber product in $\hempty$.  By applying \Cref{lem:fiber_products_fibration} to \Cref{prop:fibration} and \Cref{lem:digraph_complete}, it suffices to construct fiber products in each fiber category $\hs_\Gamma$ for $\Gamma \in \dgraph$.  Hence, we can assume $G(K_i) = \Gamma = G(H)$ for some graph $\Gamma \in \dgraph$ and that $G(p) = \id_\Gamma$. Note that to make this simplification, we are implicitly also relying on the fact that the cartesian morphisms for $G$ are diffeomorphisms on each active set; this implies that the lifting $\psi_p$ of $p$ to $\hs_\Gamma$ (using the notation of the proof of \Cref{lem:fiber_products_fibration}) is a hybrid submersion precisely when $p$ is a hybrid submersion.

For each $v \in V(\Gamma)$, let $U_{p_v} \subset M^{K_1}_v$ be an open neighborhood of $I_v^{K_1}$ for some representative of $p_v$.  Let $U_{f_v} \subset M^{K_2}_v$ be the domain of a representative of $f_v$. We define the manifold $M_v^L \subset M^{K_1}_v \times M^{K_2}_v$ to be the fiber product $U_{p_v} \times_{M^H_v} U_{f_v}$ in $\mfld$.

We define the  active set  $I_v^L \subset M_v^L$ by
\begin{linenomath*}
\[
    I_v^L = \{(x_1, x_2) \in I_v^{K_1} \times I_v^{K_2} \mid p_v(x_1) = p_v(x_2) \}
\]
\end{linenomath*}
which is just the fiber product of  $I_v^{K_1} \times_{I_v^H} I_v^{K_2}$ in $\sets$.

Similarly, we define the flow set $F^L_v \subset I^L_v$ to be the $\sets$ fiber product $F_{v}^{K_1} \times_{F_v^H} F_{v}^{K_2}$ and the continuous dynamical system $(N_v^L, X_v^L)$ as the fiber product $(N_v^{K_1}, X_{v}^{K_1}) \times_{(N_v^H, X_v^H)} (N_v^{K_2}, X_v^{K_2})$ in $\cs$, which again exists since $p_v$ is a submersion. By construction, we have $F^L_v \subset N^L_v$.

For each edge $e \in E(\Gamma)$,  the guard set $Z^L_e \subset I^L_{\src(e)}$ is given by the $\sets$ fiber product $Z_e^L = Z_e^{K_1} \times_{Z^H_e} Z_e^{K_2}$.  
Finally, we can define the reset map $r_e^L\colon Z_e^L \to I_{\tgt(e)}^L$ by $r_e^L(x_1, x_2)= (r_e^{K_1}(x_1), r_e^{K_2}(x_2))$, realizing the fiber product of the guards as a subset of the product.   

The above data specifies a hybrid system $L \in \hempty$. The canonical projections $\pi_{i,v}\colon M_v^L \to M^{K_i}_v$ define hybrid semiconjugacies $\pi_i\colon L \to K_i$.  Moreover, we have a commuting square
\begin{linenomath*}
\[
    \begin{tikzpicture}
        \node (x) at (0,0) {$L$};
        \node (y) at (3,0) {$K_2$};
        \node (z) at (0,-2) {$K_1$};
        \node (w) at (3, -2) {$H$};
        
        \draw[->, above] (x) to node {$\pi_2$} (y);
        \draw[->, left] (x) to node {$\pi_1$} (z);
        \draw[->, right] (y) to node {$f$} (w);
        \draw[->, below] (z) to node {$p$} (w);
    \end{tikzpicture},
\]
\end{linenomath*}
Since each of the components of $L$ is a fiber product of components of $K_i$ and the reset maps are simply the products of reset maps, it follows that $L$ satisfies the universal property of the fiber product.
\end{proof}

The following theorem is a translation of a result in \cite{lerman-schmitt:open_hybrid_systems} into our language and setting.  The existence of products is a consequence of \Cref{prop:hybrid_fiber_products}, but we provide a separate explicit construction here since it will be useful below in showing, for example, that the subcategory of deterministic, nonblocking systems is closed under taking products.  

\begin{prop}
\label{prop:cartesian}
The category $\hs$ is cartesian (has finite products) and cocartesian (has finite coproducts).
\end{prop}
\begin{proof}
From \Cref{ex:initial_terminal_object}, we have a terminal object $\one \in \hs$. To define binary products in $\hs$, by \Cref{lem:hempty} it suffices to define binary products in $\hempty$. Given $H_1, H_2 \in \hempty$ for $i=1,2$, we define the hybrid system $H_1 \times H_2 \in \hempty$ as follows:
\begin{itemize}
    \item $\graph G(H_1 \times H_2) = G(H_1) \boxtimes \graph G(H_2)$ as defined in \Cref{lem:digraph_complete};
    \item for each $v = (v_1, v_2) \in   V(H_1 \times H_2) =  V(H_1) \times V(H_2)$,  we let $M_v = M_{v_1} \times M_{v_2}$, $X_v = X_{v_1} \times X_{v_2}$, $I_v = I_{v_1} \times I_{v_2}$, and $F_v = F_{v_1} \times F_{v_2}$;
    \item for each $e = (e_1, e_2) \in E(H_1 \times H_2) = (E(H_1) \sqcup V(H_1)) \times (E(H_2) \sqcup V(H_2))$, we let $Z_e = Z_{e_1} \times Z_{e_2}$ and $r_e = r_{e_1} \times r_{e_2}$, where if $e_i \in V(H_i)$  we set $Z_{e_i} := F^{H_i}_{e_i}$ and $r_{e_i}: F^{H_i}_{e_i} \hookrightarrow I^{H_i}_{e_i}$ to be the inclusion.
\end{itemize} 

We also define projections $\pi_i\colon H_1 \times H_2 \to H_i$  to be the graph projection combined with the smooth projections on each continuous mode.  The semiconjugacy conditions follow from the definition of $H_1 \times H_2$.

Now suppose that $K \in \hempty$ is another hybrid system with hybrid semiconjugacies $H_1 \xleftarrow{\alpha_1} K \xrightarrow{\alpha_2} H_2$. Then we can define $\beta\colon K \to H_1 \times H_2$ by setting $\beta(v) = (\alpha_1(v), \alpha_2(v))$ for each $v \in V(K)$, and $\beta(e) = (\alpha_1(e), \alpha_2(e))$ for each $e \in E(K)$. Then $\beta_v = {\alpha_1}_v\times {\alpha_2}_v$ is a continuous semiconjugacy for each $v \in V(K)$ and we get squares of the form
\begin{linenomath*}
\[
    \begin{tikzpicture}
        \node (x) at (0,0) {$Z^K_{e}$};
        \node (y) at (3,0) {$I^K_{\tgt(e)}$};
        \node (z) at (0,-2) {$Z^{H_1 \times H_2}_{\beta(e)}$};
        \node (w) at (3, -2) {$I^{H_1 \times H_2}_{\tgt(\beta(e))}$};
        
        \draw[->, above] (x) to node {$r_e$} (y);
        \draw[->, left] (x) to node {$\beta_{\src(e)}|_{Z^K_{e}}$} (z);
        \draw[->, right] (y) to node {$\beta_{\tgt(e)}$} (w);
        \draw[->, below] (z) to node {$r_{\beta(e)}$} (w);
    \end{tikzpicture},
\]
\end{linenomath*}
(replacing $Z_{\beta(e)}$ with $F_{\beta(e)}$ if $\beta(e)$ is a vertex) which commute because the component squares commute. Since the projections are defined by the usual projections of maps, $\beta$ is the unique hybrid semiconjugacy such that $\pi_i \circ \beta = \alpha_i$ . 

The coproduct in $\hs$ is given by the coproduct of graphs with the obvious association of data to vertices and edges of the coproduct.
\end{proof}

\begin{example}
\label{ex:time_trajectory_products}
To provide an explicit example, we construct the product of two hybrid time trajectories. If $\tau = ((\tau_i)^N_{i=0}, c)$ and  $\tau' = ((\tau_j')^{N'}_{j=0}, c')$ are hybrid time trajectories, then we can form the product $\tau \times \tau' \in \hs$. Thus $G(\tau \times \tau')$ is a directed grid with vertices $v_{ij}$ for $0 \leq i < N$ and $0 \leq j < N'$ along with edges $e_{ij}^{k\ell}$ with $\src(e_{ij}^{k\ell}) = v_{ij}$ and $\tgt(e_{ij}^{k\ell}) = v_{k\ell}$ where $(k,\ell) \in \{(i, j+1), (i+1, j), (i+1, j+1)\}$. 

The active set $I^{\tau \times \tau'}_{v_{ij}}$ is the product of intervals $[\tau_i, \tau_{i+1}] \times [\tau'_j, \tau'_{j+1}]$, with the appropriate endpoint conditions for $I^{\tau \times \tau'}_{v_{Nj}}$ and $I^{\tau \times \tau'}_{v_{iN'}}$; the flow set $F^{\tau \times \tau'}_{v_{ij}} = [\tau_i, \tau_{i+1}) \times [\tau'_j \times \tau'_{j+1})$ with vector field $X_{v_ij}^{\tau \times \tau'} = (\frac{d}{dt}, \frac{d}{dt})$; and the guards 
\begin{linenomath*}
\[
    Z_{e_{ij}^{k\ell}} = 
    \begin{cases}
        [\tau_i, \tau_{i+1}) \times \{\tau'_{j+1}\} & \text{if } (k, \ell) = (i, j+1)\\
        \{\tau_{i+1}\} \times [\tau'_j, \tau'_{j+1}) & \text{if } (k, \ell) = (i+1, j)\\
        \{\tau_{i+1}\} \times \{\tau'_{j+1}\} & \text{if } (k, \ell) = (i+1, j+1)
    \end{cases}
\]
\end{linenomath*}
with resets $r_{e_{ij}^{k\ell}}\colon Z_{e_{ij}^{k\ell}} \to I^{\tau \times \tau'}_{v_{k\ell}}$ sending
\begin{linenomath*}
\[
    (s,t) \mapsto
    \begin{cases}
        (s, r^{\tau'}_{e_j}(t)) & \text{if } (k, \ell) = (i, j+1)\\
        (r^{\tau}_{e_i}(s), t) & \text{if } (k, \ell) = (i+1, j)\\
        (r^{\tau}_{e_i}(s), r^{\tau'}_{e_j}(t)) & \text{if } (k, \ell) = (i+1, j+1).
    \end{cases}
\] 
\end{linenomath*}

Now we can easily see that the flow sets and guards are pairwise disjoint, so $\tau \times \tau'$ is a deterministic system. \qed
\end{example}

\begin{cor} \label{cor:dn_prod}
The category $\hsdn$ is cartesian and cocartesian.
\end{cor}
\begin{proof}
 Let $H_1, H_2 \in \hsdn$ be deterministic, nonblocking hybrid systems.  It is clear that their $\hempty$-product $H:= H_1 \times H_2$ is nonblocking.  
 To check determinism, first we note that the disjointness of guards and flow sets condition of \Cref{def:deterministic} follows from identity 
 \begin{linenomath*}
 \[
 \{F_v\}_{v \in V(H)} \cup \{Z_e\}_{e \in E(H)} = 
 \{ U_1 \times U_2 \mid U_i \in \{F_v\}_{v \in V(H_i)} \cup \{Z_e\}_{e \in E(H_i)} \}
 \]
 \end{linenomath*}
and the disjointness of the corresponding guard and flow sets for $H_1, H_2$.   
To construct the unique maximal $X_{(v_1,v_2)}$-integral curve starting at any point $(x_1, x_2) \in F_v^{H_1} \times F_v^{H_2}$ for any $v_i \in H_i$, 
first let $J_i\colon [0,T_i) \to F_{v_i}^{H_i}$ be the unique maximal $X_{v_i}$-integral curves (which must exist by the determinism of each $H_i$).
Then $(J_1 \times J_2)_{|[0, \min(T_1, T_2))}$ is the unique maximal $X_{(v_1,v_2)}$-integral curve starting at $(x_1, x_2)$.

Since the pruning functor $P \colon \hempty \to \hs$ of \Cref{lem:hempty} preserves executions and the disjointness of guards and flow sets condition, it follows that $\hsdn$ is cartesian. The cocartesian proof is straightforward.
\end{proof}

Lastly, the following lemma provides a convenient way to reason about points of deterministic hybrid systems using semiconjugacies.

\begin{lemma}
\label{lem:representability}
Let $U\colon \tops \to \sets$ be the forgetful functor. The functor $U \circ I\colon \hs \to \sets$ is represented by the hybrid system $N \in \hs$ where $G(N) = {\bullet}$, $M_{\bullet} = I_{\bullet} = *$, and $F_\bullet = \emptyset$.
\end{lemma}

\section{Hierarchical composition}  \label{sec:hierarchical}
In this section, we define hierarchical compositions of hybrid systems via fiber products of template-anchor pairs.  To construct a template-anchor pair, one starts with (i) a template model describing the essential dynamics of the system and (ii) a more detailed anchor model of the same system  \cite{Full_Koditschek_1999}.  We would like to compute using template models but also have a guarantee that the template-based computations also apply to the more detailed anchor model.  That is, a template-anchor pair should form a simple hierarchy in which the dynamics of the template approximately determine the dynamics of the anchor.  To accomplish this, one embeds of the template model into the anchor model, such that the image of the embedding is attracting in the sense that it is an isolated positively invariant set, as described below. Conceptually, this is analogous to the setting of  \cite{Kvalheim_Revzen_2016} where normally hyperbolic invariant manifolds are explored for classical systems.   

\subsection{Attracting sets}
Before defining any notion of attracting set, we begin with a basic property: positive invariance.

\begin{defn}
Let $H \in \hs$ be a hybrid system.  We say that a set $A \subset I(H)$ is {\bf positively invariant} if $\chi(t) \subset A$ for all $x \in A$, $\chi \in \infexc_H(x)$, and $t \geq 0$. 
\end{defn}

The following definition is a generalization of Hurley's definition of attractor for a flow on a compact manifold \cite{hurley1982attractors} and a definition due to Milnor in the discrete case \cite{milnor1985concept}.    

\begin{defn}\label{def:iso_inv}
Let $H$ be a hybrid system, and let $A \subset I(H)$ be positively invariant.  A (forward) {\bf isolating neighborhood} $N \subset I(H)$ for $A$ is a neighborhood of $A$ such that
\begin{linenomath*}
\[
    A = \bigcap_{t \geq 0} \left\{\chi(t) \mid x \in N \text{ and } \chi \in \infexc_H(x)\right\}.
\]
\end{linenomath*}
If such an $N$ exists, we call $A$ an {\bf attracting set}.
\end{defn}

Trapping regions are particularly well-behaved examples of isolating neighborhoods (as shown below). For a continuous system, a trapping region is a positively invariant subset of the state space whose closure flows into its interior in uniform finite time. The follow hybridization of this idea is a faithful generalization, recovering the continuous notion if the hybrid system $H$ has a single continuous mode $v$ and no resets.  

\begin{defn} \label{def:trapping}
A {\bf hybrid trapping region} for a hybrid system $H$ is a topological subspace $W \subset I(H)$ such that
\begin{enumerate}
    \item for all $w \in W$, the set $\infexc_H(w)$ is nonempty; 
    \item $W$ is positively invariant; and
    \item there exists a finite time $T > 0$ such that for each $x \in \overline{W}$ and $\chi \in \infexc_H(x)$, we have $\chi(t) \subset W^\circ$ for all $t \geq T$.
\end{enumerate}
\end{defn}

\begin{remark}
To recover the notion of trapping region for the discrete dynamical system determined by a function $f$ on a manifold $X$, we can define a hybrid system $H$ whose graph is a path $(v_i)_{i=1}^\infty$ and $M_{v_i} = I_{v_i} = X \times [i-1, i]$, $F_{v_i} = X \times [i-1, i)$, $Z_{e_i} = X \times \{i\}$, and $r_{e_i} = f \times \id$ where $e_i$ is the reset from $v_i$ to $v_{i+1}$.   The vector field on each flow set is given by $X_{v_i} = 0 \times \frac{d}{dt}$. Then a trapping region $W \subset X$ for the discrete system $(f,X)$ corresponds to a trapping region $\bigcup_i W \times [i, i+1] \subset I(H)$. 
\end{remark}

\begin{defn}
Let $H$ be a hybrid system, and let $W \subset I(H)$ be a trapping region.  The {\bf trapped attracting set} for $W$ is the set $A \subset I(H)$ given by
\begin{linenomath*}
\[
    A = \bigcap_{t \geq 0} \left\{\chi(t) \mid x \in W \text{ and } \chi \in \infexc_H(x)\right\}.
\]
\end{linenomath*}
\end{defn}

\begin{prop}
Let $H \in \hs$ be a hybrid system. Let $A \subset I(H)$ be the trapped attracting set for a trapping region $W \subset I(H)$.  Then $A$ is positively invariant. 
In particular, $A$ is an attracting set with isolating neighborhood $W$.
\end{prop}
\begin{proof}
Let $x \in A$ and $s \geq 0$.  We want to show that if $\chi_x \in \infexc(x)$, then $\chi_x(s) \subset A$.

Let $u \in \chi_x(s)$.  Then for all $t \geq 0$, there exist $y \in W$ and $\chi_y \in \infexc(y)$ such that $x \in \chi_y(t)$.
Let $s_0 \in \tau(s)$.  Then $z := \chi_y(s_0)$ lies in $W$ by the positive invariance of $W$.   Let $\xi \in \infexc_\tau(s_0)$ be the unique fundamental infinite execution of $s_0$ in $\tau$, and let $\xi \colon \sigma \to \tau$ be a corresponding semiconjugacy.  Then $\chi_y \circ \xi \in \infexc_H(z)$, and $u \in (\chi_y \circ \xi)(t)$.   Thus, $u \in A$.  Therefore, $\chi_x(s) \subset A$.
\end{proof}

In both the discrete and continuous case, every compact attracting set is actually a trapped attracting set for a trapping region \cite{hurley1982attractors, Milnor:2006}.  As the following example shows, this is not true for hybrid systems.

\begin{example} \label{ex:trap}
Let $\mathbb{S}^1$ be the unit circle embedded in $\mathbb{R}^2$ with center at $(0,1)$. Then let $H \in \hsdn$ be the hybrid system given by
\begin{itemize}
    \item $G = \overset{v}{\bullet} \overset{e}{\to} \overset{w}{\bullet}$
    \item $(M_v, X_v) = (\mathbb R, (x+1)(2 - x) \frac{d}{dx})$
    \item $I_v = [-1, 2]$
    \item $F_v = I_v \setminus \{0\}$
    \item $Z_e = \{0\}$
    \item $M_w = I_w = F_w = \mathbb{S}^1$
    \item $X_w = -\nabla h$ where $h\colon \mathbb{S}^1 \to \mathbb{R}$ is given by $h(x,y) = y$
    \item $r_e(0) = (1,1)$
\end{itemize}
Then $(-1,1) \cup \mathbb{S}^1$ is an isolating neighborhood for the attracting set $A := (-1, 0] \cup \mathbb{S}^1$, but there is no positively invariant isolating neighborhood for $A$.  In particular, there is no trapping region for $A$. \qed
\end{example}

The preceding example also demonstrates the necessity of the positive invariance condition in \Cref{def:iso_inv}.  Since $\chi(0) = \{0, y\}$ for each $\chi \in \infexc_H(0)$, we have
\begin{linenomath*}
\[
    \{0, y\} = \bigcap_{t \geq 0} \left\{\chi(t) \mid x \in (-1,1) \text{ and } \chi \in \infexc_H(x)\right\},
\]
\end{linenomath*}
but the set $\{0, y\}$ is not positively invariant. 

\subsection{Templates and anchors via spans}
\label{sec:templates_anchors}

In the continuous case, the underlying data of a template-anchor pair is simply an embedding of the template system into the anchor system.  In the hybrid case, one often wants template embeddings to map smooth systems to hybrid systems with nontrivial discrete dynamics.  This type of template-anchor relationship does not constitute a simple hybrid semiconjugacy, but rather a hybrid semiconjugacy from a subdivision of the template system into the anchor system. Here we develop the necessary machinery to allow for this flexibility in representing hybrid templates and anchors.

\subsubsection{Subdivisions}

To define subdivisions of arbitrary hybrid systems, we begin by defining subdivisions of hybrid time trajectories, which we call ``refinements.''

\begin{defn}
A {\bf refinement} of a hybrid time trajectory $\tau$ is a hybrid surjective submersion $\sigma \to \tau$ where $\sigma$ is also a hybrid time trajectory.
\end{defn}

\begin{example}
Let $\chi\colon \tau \to H$ be any execution. Let $\chi^*\colon \tau^* \to H$ be the associated fundamental execution.  Then there exists a refinement $p\colon \tau \to \tau^*$ such that $\chi^* \circ p = \chi$. \qed
\end{example}

The following lemma shows that refinement preserves the time $t$ points of a hybrid time trajectory. 

\begin{lemma}
\label{lem:ss_time}
If $\xi\colon \sigma \to \tau$ is a refinement, then $\tau(t) = \xi(t)$ for all $t$.
\end{lemma}
\begin{proof}
Since $I(\xi)$ is surjective, it suffices to show that $\xi(t) \subset \tau(t)$ for all $t$.    Let $t \geq 0$ and $x \in \sigma(t)$.  Let $(v_i)_{i=1}^N$  and $(e_i)_{i=1}^{N-1}$ be the sequences of vertices and edges for the path $G(\sigma)$, respectively.  Let $v_i$ be the continuous mode such that $x \in I_{v_i}$.  
We proceed by induction on $i$.

Since $I(\xi)$ is surjective, it must map the starting point of $\sigma$ to the starting point of $\tau$.   It follows from the continuous semiconjugacy condition that $\xi(x) \in \tau(t)$ for all $x \in I_{v_1}$.

Now suppose that $\xi(x) \in \tau(t)$ if $x \in I_{v_i}$ for some $i$.  We want to show that the same thing holds if $x \in I_{v_{i+1}}$.  If $x$ is the left endpoint of $I_{v_{i+1}}$, then $y := r_{e_i}^{-1}(x) \in \sigma(t)$.   By the induction hypothesis, $\xi(y) \in \tau(t)$.   Moreover, we know that $\xi(r_{e_i})$ is either a reset map or an identity map for a continuous mode of $\tau$.  Thus, $\xi(x) = \xi(r_{e_i}(y)) \in \sigma(t)$. Using the continuous semiconjugacy condition, it follows that $\xi(x) \in \tau(t)$ for all $x \in I_{v_{i+1}}$.
\end{proof}

The next proposition, which follows directly from the definitions, justifies the choice of the word ``refinement.'' 

\begin{prop}
Let $\tau = ((\tau_j)_{j = 0}^N, c_\tau)$ be a hybrid time trajectory.  There exists a bijection between refinements $\sigma \to \tau$ and pairs $(\sigma, k)$ where $\sigma = ((\sigma_j)_{j=0}^M, c_\sigma)$ is a hybrid time trajectory and $k\colon \{0, 1, \ldots, N\} \to \{0, 1, \ldots, M\}$ is an increasing map such that
\begin{itemize}
    \item $k(0) = 0$, 
    \item $k(N) = M$ if $N < \infty$,
    \item $\tau_j = \sigma_{k(j)}$ for $0 \le j \le N$, and
    \item $c_\tau = c_\sigma$.
\end{itemize}
\end{prop}

\begin{defn}\label{def:hyb_sub}
A {\em \bf hybrid subdivision} of a hybrid system $H \in \hs$ consists of a hybrid system $S \in \hs$ and a hybrid submersion $p \colon S \to H$ such that for every hybrid execution $\chi \colon \tau \to H$ there exists a refinement $\xi \colon \sigma \to \tau$ and an execution $\widetilde{\chi} \colon \sigma \to S$ such that 
\begin{linenomath*}
\[
    \begin{tikzpicture}
        \node (x) at (0,0) {$\sigma$};
        \node (y) at (3,0) {$S$};
        \node (z) at (0,-2) {$\tau$};
        \node (w) at (3, -2) {$H$};
        
        \draw +(.25,-.75) -- +(.75,-.75)  -- +(.75,-.25);

        \draw[->, above] (x) to node {$\widetilde{\chi}$} (y);
        \draw[->, left] (x) to node {$\xi$} (z);
        \draw[->, right] (y) to node {$p$} (w);
        \draw[->, below] (z) to node {$\chi$} (w);
    \end{tikzpicture}
\]
\end{linenomath*}
is cartesian ($\sigma$ is a fiber product).
\end{defn}

Informally, a hybrid subdivision of a system $H$ can split a continuous mode of $H$ into multiple modes connected with resets to preserve the dynamics. For example, to formalize the relationship between a periodic continuous orbit and a hybrid periodic orbit, it is useful to model the relationship between a periodic flow on $\mathbb S^1$ and the periodic hybrid dynamics on a broken $\mathbb S^1$ connected by resets.

\begin{example} \label{ex:slice}
Slicing a continuous dynamical system via a codimension-$1$ submanifold transverse to the flow gives a subdivision, assuming no fundamental Zeno executions appear in the subdivided system.  Simply duplicate the boundary, and define a new reset from the copy of the boundary with outward pointing vector to the other copy via the identity function on the boundary. Then executions pull back with boundary points doubled---this gives a refinement with an extra jump point at every place where the execution hits the submanifold. \qed
\end{example}

\begin{example}
For a discrete system ({\em i.e.}, a hybrid system whose flow sets are all empty), subdivision of continuous modes is even simpler. To subdivide a mode, pick a partition $\{I_{v,j}\}_j$ of $I_v$  and define $M_{v,j} = M_v$. The new guards are the intersections of the guard sets in $I_v$ with each $I_{v,j}$.  The reset maps into and out of each $I_{v,j}$ are the restrictions of the corresponding reset maps into and out of $I_v$. 
\end{example}

\begin{example}
We can also sequentially subdivide reset maps. Let $H$ be a hybrid system with a reset map $r_e\colon Z_e \to I_{w}$ for some edge $e\colon v \to w$.  Let $f\colon Z_e \to A$ and $g\colon A \to I_{w}$ be maps such that $r_e = g \circ f$, where $A$ is any subset of a manifold $M$.

We define $S$ to be the hybrid system given by replacing the edge $e$ with two edges $ v \xrightarrow{e_f} u$ and $u \xrightarrow{e_g} w$ where $u$ is a new node with $(M_u, X_u) = (M,0)$, $Z_u = I_u = A$, and $F_u = \emptyset$.  Then $p\colon S \to H$ defined by $p_u = g$ and the identity on the other vertices is a subdivision.
\end{example}

\begin{example}
A subdivision of a deterministic system is not necessarily deterministic in the sense of \Cref{def:deterministic}. For example, let $H$ be the smooth system $(\mathbb{R}, \frac{d}{dt})$.  Let $S$ be the system given by the graph $G(S) = \overset{v}{\bullet} \overset{e}{\to} \overset{w}{\bullet}$ where $(M_v, X_v) = (\mathbb{R}, \frac{d}{dt}) = (M_w, X_w)$, $F_v = (-\infty, 0]$, $Z_e = \{0\}$, and $F_w = [0, \infty)$.    Then the semiconjugacy $p\colon S \to H$ induced by the identity map on $\mathbb{R}$ is a subdivision, but $S$ is not deterministic in the sense of \Cref{def:deterministic}. \qed
\end{example}

\begin{prop} \label{prop:simp_hyb}
Every refinement is a hybrid subdivision. 
\end{prop}
\begin{proof}
Let $p\colon \rho \to \eta$ be a refinement, and let $\chi\colon \tau \to \eta$ be an execution.  By the proof of \Cref{prop:hybrid_fiber_products}, we have a cartesian square
\begin{linenomath*}
\[
    \begin{tikzpicture}
        \node (x) at (0,0) {$P(L)$};
        \node (y) at (3,0) {$\rho$};
        \node (z) at (0,-2) {$\tau$};
        \node (w) at (3, -2) {$\eta$};
        
        \draw[->, above] (x) to node {$P(\pi_1)$} (y);
        \draw[->, left] (x) to node {$P(\pi_2)$} (z);
        \draw[->, right] (y) to node {$p$} (w);
        \draw[->, below] (z) to node {$\chi$} (w);
    \end{tikzpicture}
\]
\end{linenomath*}
such that $(L, \pi_1, \pi_2)$ forms the fiber product in $\hempty$.  

Let $G(\tau)$ be the path on the vertices $(v_i)_{i=0}^{M-1}$, letting $e_i$ denote the edge from $v_i$ to $v_{i+1}$. Similarly, let $G(\rho)$ be the path $(w_j)_{j=0}^{N-1}$, letting $f_j$ denote the edge from $w_j$ to $w_{j+1}$.  Let $s \in I(\eta)$ be the starting point of the execution $\chi$. Then if 
$(v_i,w_j) \in V(L) = V(\tau) \times_{V(\eta)} V(\rho) \subset V(\tau) \times V(\rho)$, we have 
\begin{itemize}
    \item $M_{v_i, w_j}^L = M_{v_i}^\tau {\times}_{\chi,p} M_{w_j}^\rho = \mathbb{R}$,
    \item $I_{v_i, w_j}^L = I_{v_i}^\tau {\times}_{\chi,p} I_{w_j}^\rho = I_{v_i}^\tau \cap  (I_{w_j}^\rho - s)$, and
    \item $F_{v_i, w_j}^L = F_{v_i}^\tau {\times}_{\chi,p} F_{w_j}^\rho = F_{v_i}^\tau \cap  (F_{w_j}^\rho - s)$.
    \item $X_{v_i, w_j}^L = \frac{d}{dt}$
\end{itemize}
If $(e_i, f_j) \in E(L) = (E(\tau) \times_{\chi,p} E(\rho)) \sqcup (V(\tau) \times_{\chi,p} E(\rho)) \sqcup (E(\tau) \times_{\chi,p} V(\rho))$, then
\begin{itemize}
    \item if $(e_i, f_j) \in  (E(\tau) \times_{\chi,p} E(\rho))$, then $Z_{e_i, f_j}^L = Z_{e_i} \cap (Z_{f_j} - s)$
    \item if $(e_i, f_j) \in  (V(\tau) \times_{\chi,p} E(\rho))$, then $Z_{e_i, f_j}^L = F_{e_i}  \cap (Z_{f_j} - s)$
    \item if $(e_i, f_j) \in  (E(\tau) \times_{\chi,p} V(\rho))$, then $Z_{e_i, f_j}^L = Z_{e_i} \cap (F_{f_j} - s)$
\end{itemize}
In each case, the reset map $r_{e_i,f_j}^L \colon Z_{e_i,f_j}^L \to I_{\tgt(e_i), \tgt(f_j)}$ is given by the inclusion map of underlying sets. One can check that $P(L)$ is a hybrid time trajectory, and $P(\pi_2)$ is a hybrid surjective submersion.
\end{proof}

The next lemma further demonstrates that the dynamics on a subdivision are essentially the same as that of the underlying system. 

\begin{lemma}
\label{lem:subdivision_executions_fibers}
Let $p\colon S \to H$ be a hybrid subdivision.  Then for any execution $\chi\colon \tau \to H$ and fiber product $\tilde \chi\colon \sigma \to S$ of $\chi$ along $p$, we have $\tilde \chi(t) = p^{-1}(\chi(t))$ for every $t \geq 0$.
\end{lemma}
\begin{proof}
Let $x \in p^{-1}(\chi(t))$.  Then there exists $u \in \tau(t)$ such that $x \in p^{-1}(\chi(u))$.  From the definition of subdivision, we have the following commutative diagram 
\begin{linenomath*}
\[
    \begin{tikzpicture}
        \node (t) at (0,0) {$\tau$};
        \node (s) at (4,2) {$S$};
        \node (h) at (4,0) {$H$};
        \node (si) at (0,2) {$\sigma$};
        \node (o) at (-1, 3) {$N$};
        
        \draw[->,below] (t) to node {$\chi$} (h);
        \draw[->, right] (s) to node {$p$} (h);
        \draw[->, above right] (o) to node {$\hat v$} (si);
        \draw[->, left, out=270, in=160] (o) to node {$\hat u$} (t);
        \draw[->, above, out=0, in=140] (o) to node {$\hat x$} (s);
        \draw[->, above] (si) to node {$ \tilde \chi $} (s);
        \draw[->, left] (si) to node {$ \xi $} (t);
    \end{tikzpicture}
\]
\end{linenomath*}
where $N$ is the representing object defined in \Cref{lem:representability}, and $v \in \sigma(t)$.  Thus, $x \in \widetilde \chi(t)$.  Hence, $p^{-1}(\chi(t)) \subset \tilde \chi(t)$.  

Conversely, given $x \in \tilde \chi (t)$, there exists a $ v \in \sigma(t)$ such that $x = \chi(v)$.  \Cref{lem:ss_time} implies that $y := \xi(v)$ lies in $\tau(t)$.  Thus $p(x) \in \chi(t)$. Hence, $\tilde \chi(t) \subset p^{-1}(\chi(t))$. 
\end{proof}

\begin{lemma}
\label{lem:subdivision_executions_resets}
If $H$ is nonblocking, deterministic and $p\colon S \to H$ is a hybrid subdivision and $x \in M_v^H$, then $p^{-1}(x)$ is finite and totally ordered by a sequence of resets. In particular, each $y \in p^{-1}(x)$ (except possibly the last one) is contained in a guard set in $S$. 
\end{lemma}
\begin{proof}
Since $H$ is nonblocking and deterministic, there is exactly one fundamental execution $\chi_x \in \infexc_H(x)$, which pulls back to an execution $\tilde \chi_x \in \infexc_S$. Now if $y \in p^{-1}(x)$, we can use the representing object $N$ from \Cref{lem:representability} and the universality of the fiber product to get the diagram
\begin{linenomath*}
\[
    \begin{tikzpicture}
        \node (x) at (0,0) {$\sigma$};
        \node (y) at (3,0) {$S$};
        \node (z) at (0,-2) {$\tau$};
        \node (w) at (3, -2) {$H$};
        \node (n) at (-1, 1) {$N$};
        
        \draw +(.25,-.75) -- +(.75,-.75)  -- +(.75,-.25);

        \draw[->, above, out=0, in=140] (n) to node {$\hat y$} (y);
        \draw[->, left, out=270, in=160] (n) to node {$\hat 0$} (z);
        \draw[->, above right] (n) to node {$\hat z$} (x);
       
        \draw[->, above] (x) to node {$\tilde{\chi}_x$} (y);
        \draw[->, left] (x) to node {$\xi$} (z);
        \draw[->, right] (y) to node {$p$} (w);
        \draw[->, below] (z) to node {$\chi_x$} (w);
    \end{tikzpicture}
\]
\end{linenomath*}
where $\xi$ is a refinement. Since $\xi^{-1}(0) = \sigma(0)$ and $x \in \chi_x(0)$, we see that $p^{-1}(x)$ consists of repeated resets at time $0$, and has an order induced by $\sigma$. Since $H$ is nonblocking, $\chi^t_x$ exists for all $t \geq 0$, and so there must only be finitely many $0$ points. 
\end{proof}

\begin{thm}
\label{thm:sudvision_properties}  
If $p\colon S \to H$ is a hybrid subdivision with $S, H \in \hsdn$, then
\begin{enumerate}
    \item[(i)] $p\colon G(S) \to G(H)$ is a graph epimorphism;
    \item[(ii)] $I(p)$ is surjective;
    \item[(iii)] $p_v\colon I^S_v \to I^H_{p(v)}$ is an injective local diffeomorphism.
\end{enumerate}
\end{thm}

\begin{proof}
Statements (i) and (ii) follow from the fact that the definition of subdivision guarantees the existence of an execution $\tilde\chi_x$ of $S$ pulling back $\chi_x \in \infexc_H(x)$, for every $x \in I(H)$. More precisely, given $x \in I^H_v$, we have that $x \in (p \circ \tilde\chi_x)(0)$, and so necessarily $x \in \im(I(p))$ and hence $v \in \im(p)$. To see that $p$ is surjective on edges, let $e \in E(H)$ and $x \in Z^H_e \subset I^H_v$. Then $x\in \im(I(p))$ as above. Because $S, H$ are deterministic and $p_u(F^S_u) \subset F^H_u$ for all $u \in V(S)$, there must be some $\tilde e \in E(S)$ with $x \in p_{\src(\tilde e)}(Z^S_{\tilde e})$. But $p_{\src(\tilde e)}(Z^S_{\tilde e}) \subset Z^H_{p(\tilde e)}$ and so $p(\tilde e) = e$. 

For (iii), suppose that $p_v$ is not injective for some $v \in V(S)$. Then there exist $y_1 \neq y_2 \in p_v^{-1}(x) \subset p^{-1}(x) \cap I^S_v$. Let $z$ be the initial point of $\chi_x$. Then $(z,y_1)$ and $(z, y_2)$ are nonequal initial points of $\tilde\chi_x$, a contradiction. Since $p_v$ is a submersion, it is also open and hence a local diffeomorphism.
\end{proof}

\begin{lemma}
\label{lem:subd_compose}
If $p_1 \colon S \to H$ and $p_2 \colon H \to K$ are hybrid subdivisions, then $p_2 \circ p_1$ is a hybrid subdivision.
\end{lemma}
\begin{proof}
The hybrid surjective submersion condition is clear.  The subdivision condition follows from the fact that the vertical composition of cartesian squares is cartesian.
\end{proof}

\subsubsection{Template-anchor pairs}
We are now able to define our notion of templates and anchors consistent with the ideas described in \cite{Full_Koditschek_1999} and suitable for the hybrid setting. 

\begin{defn} \label{def:template-anchor}
A {\bf template-anchor pair} is a span $T \xleftarrow{p} S \xrightarrow{i} A$ such that 
\begin{enumerate}[(i)]
    \item  $p$ is a hybrid subdivision;
    \item $i$ is a hybrid embedding;
    \item $i(S)$ is an attracting set in $A$. 
\end{enumerate} 
\end{defn}

As an illustration of our template-anchor framework, the following example demonstrates how to anchor a template limit cycle in a vertical  hopper control system, translating the work of \cite{De_Koditschek_2015} into our categorical language. We use the ubiquitous spring-loaded inverted pendulum (SLIP) to model an abstract leg.  Despite its simplicity, the SLIP model has been instrumental in analyzing limbed locomotion in animals \cite{Holmes_Full_Koditschek_Guckenheimer_2006}, as well as in synthesizing control laws for robotic legs \cite{Raibert_1986}.

\begin{example}[Vertical hopper \cite{De_Koditschek_2015}]
\label{ex:vhop}
The aim of this example is to demonstrate how to anchor a template limit cycle in a periodic vertical hopper control policy. In our mathematical parlance, this amounts to constructing a template-anchor pair $(\mathbb{S}^1, Y) \leftarrow S \rightarrow H_\hop$, for some period-defining template vector field $Y$ and anchoring control hybrid system $H_\hop$.  We begin by defining $H_\hop$ (\Cref{fig:hhop}).   

The graph $G(H_\hop)$ has a single continuous mode $v$ corresponding to stance and a single reset $e$ corresponding to the integration of the ballistic trajectory between takeoff and landing.   To describe the stance dynamics, we start with a driven, damped harmonic oscillator:
\begin{linenomath*}
\[
    \ddot{\delta} + 2 \omega \beta \dot{\delta} + \delta = \lambda,
\]
\end{linenomath*}
where $\delta$ is the spring deflection, $\beta$ the damping coefficient,  $\omega$ the natural frequency, and $\lambda$ the driving force.   Setting $x_1 := \delta$ and $x_2 := \dot{\delta}/\omega$, we define the ambient manifold to be $M_v := \mathbb R^2 \setminus \{0\}$  
and the stance active set to be 
\begin{linenomath*}
\[
    I_{v} := \{ x_1 \le 0 \} \subset \mathbb{R}^2 \setminus \{0\}
\]
\end{linenomath*}
Under this change of variables, the differential equation above gives the vector field
\begin{linenomath*}
\begin{equation}
\label{eqn:hop}
X_v := \begin{pmatrix} \dot{x}_1 \\ \dot{x}_2 \end{pmatrix}
= 
\begin{pmatrix}
\omega x_2 \\
\lambda/\omega - \omega x_1  - 2 \beta \omega x_2
\end{pmatrix}
\end{equation}
\end{linenomath*}

For ease of exposition, we use a slight simplification of the control policy in \cite{De_Koditschek_2015}: 
\begin{linenomath*}
\[
    \lambda := \frac{k_t x_2}{\|x\| },
\]
\end{linenomath*}
where $k_t$ is a constant that can be tuned to maintain a given jump height.  
The guard set is $Z_e := \{(x_1, x_2) \in M_v \mid x_1 = 0 \text{ and } x_2 \geq 0\}$. The reset map $r_e$ is  given by $(0, x_2) \mapsto (0, -x_2)$, that is, the flight phase simply reverses the velocity.  
The flow set is $F_v := I_v \setminus Z_e$. 

By analyzing the Poincar\'e map associated to the guard set $Z_e$, one can show that that $H_\hop$ contains a unique attracting limit cycle \cite{De_Koditschek_2015}. However, we will prove the same result using a more elementary comparison of $H_\hop$ to a smooth system on the punctured plane (for which it is easy to prove the existence of a limit cycle), and as a by-product explicitly define a template-anchor pair for the $\mathbb{S}^1$-limit cycle  inside of $H_\hop$.    

\begin{figure*}[t!]
    \centering
    \begin{subfigure}[t]{0.45\textwidth}
        \centering
        \includegraphics[height=1.5in]{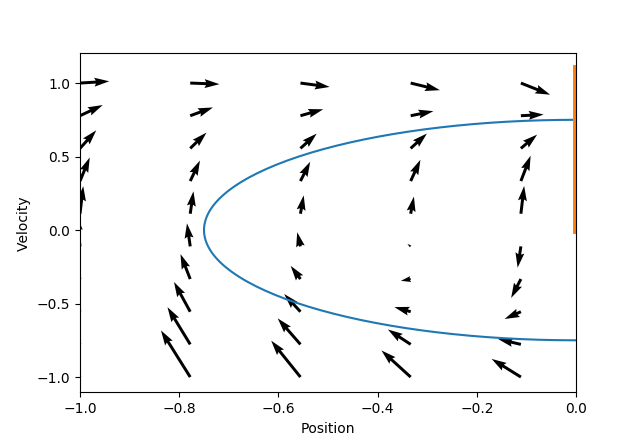}
        \caption{The continuous mode of $H_\hop$  with its guard $Z_v$ in orange and its attracting limit cycle $g(S)$ in blue.  The reset map reflects any point in the guard through the origin. }
        \label{fig:hhop}
    \end{subfigure}%
    \hfill
    \begin{subfigure}[t]{0.45\textwidth}
        \centering
        \includegraphics[height=1.5in]{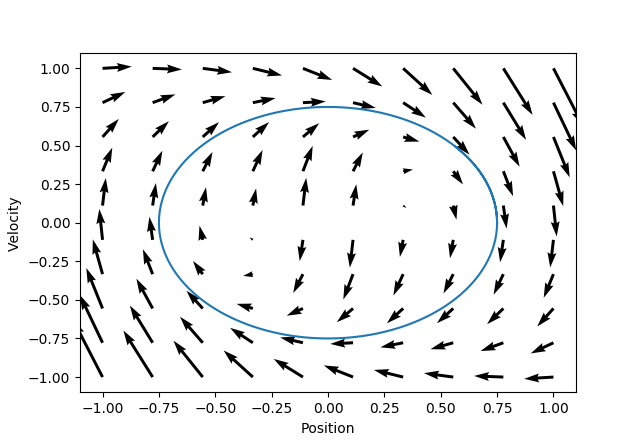}
        \caption{The continuous system $L$ with its attracting limit cycle $f(\mathbb{S}^1)$ in blue}
        \label{fig:l}
    \end{subfigure}
    
    \begin{subfigure}[t]{0.9\textwidth}
        \centering
        \includegraphics[height=1.5in]{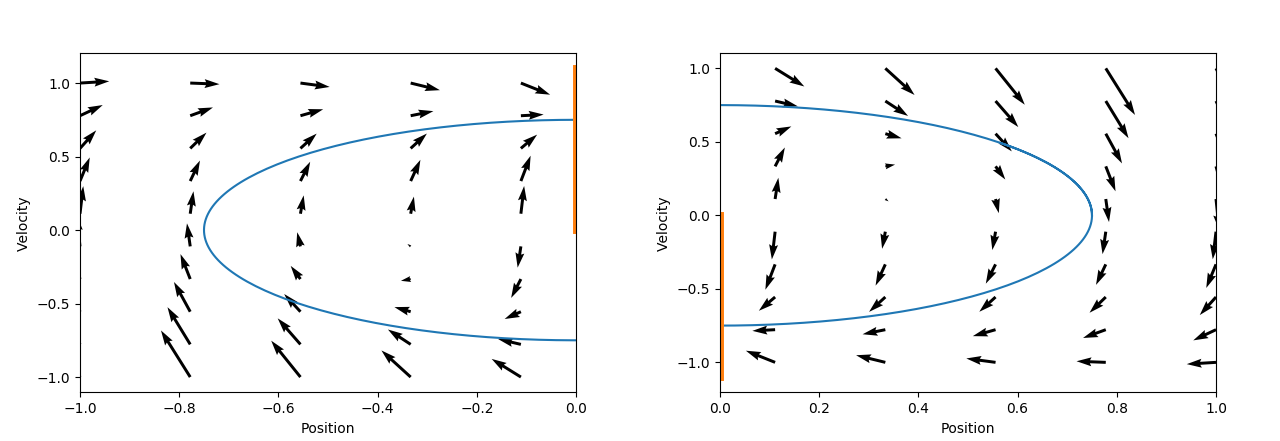}
        \caption{The two modes of $K$ with the suggestive change of coordinates $x \mapsto -x$ for the mode on the right.  Guards are in orange and the attracting limit cycle $e(C)$ in blue.  The resets map a point in a guard to the point with the same coordinates in the other mode.}
        \label{fig:k}
    \end{subfigure}
    \caption{Comparing $H_\hop$, $K$, and $L$}
    \label{fig:comparison}
    
\end{figure*}

As an intermediate step between $H_\hop$ and this continuous system, we first construct the following ``double cover'' hybrid system $K$  (\Cref{fig:k}). We define the graph 
\begin{linenomath*}
\[
    G(K) := {v_1} \,\bullet \rightleftarrows \bullet \, {v_2},
\]
\end{linenomath*}
which is a double cover of $G(H_\hop)$.   Each continuous mode of $K$ is defined to be identical to the continuous mode of $H_\hop$.  The resets of $K$ are also identical to the reset of $H$ except each maps from $Z_{v_i} \to I_{v_{1-i}}$ instead of $Z_v \to I_v$.  We define $s \colon K \to H$ to be the two-fold covering semiconjugacy.

We can now define $L$ to be the smooth system  $(\mathbb{R}^2 \setminus \{0\}, X)$ where the vector field $X$ is defined by the same formula as \Cref{eqn:hop} (see \Cref{fig:l}). Intuitively, $L$ corresponds to gluing $K$ together along its reset maps.  More formally, since
\begin{equation} \label{eqn:odd} 
    \dot{x}_{|x = (-x_1, -x_2)} = \left(-\omega x_2, \omega x_1 - \left(\frac{k_t}{\|x\|} - 2 \beta \omega\right) x_2\right)^T = - \dot{x}_{|x = (x_1, x_2)},
\end{equation} the map $p\colon K \to L$ defined by $p_{v_1}(x_1, x_2) = (x_1, x_2)$ and $p_{v_2}(x_1, x_2) = -(x_1, x_2)$ is a subdivision (an instance of \Cref{ex:slice}).  

Since $L$ is a classical dynamical system, it is much easier to explicitly prove the existence of a limit cycle.  Indeed, as in the proof of Proposition 1 of \cite{De_Koditschek_2015},  for each $x \in \mathbb{R}^2 \setminus \{0\}$ we have
\begin{linenomath*}
\[
    x \cdot \dot{x} = x_2^2 \left( \frac{k_t}{\omega \|x\|} -2 \beta \omega  \right),
\]
\end{linenomath*}
which is zero when $\|x\| = \frac{k_t}{2 \beta \omega^2}$.  
In addition, $x \cdot \dot{x} \geq 0$ if $\|x\| < \frac{k_t}{2 \beta \omega^2}$, and $x \cdot \dot{x} \leq 0$ if  $\|x\| > \frac{k_t}{2 \beta \omega^2}$.  Since $\frac{d}{dt} \|x\|^2 = 2 (x \cdot \dot{x})$, it follows from LaSalle's theorem \cite{lasalle1968stability,lasalle1968stability} that the set $\|x\| = \frac{k_t}{2 \beta \omega^2}$ forms an attracting set in $L$.

We define $f\colon \mathbb{S}^1 \to \mathbb{R}^2 \setminus \{0\}$ to be the embedding corresponding to this attracting limit-cycle, namely $f(x) := \frac{k_t}{2 \beta \omega^2} x$.   Then $f$ defines an semiconjugacy $(\mathbb{S}^1, f^*X) \to L$,  where $f^*X$ is the pullback vector field, which is well-defined since $f(\mathbb{S}^1)$ is invariant in $L$.   

To anchor an $\mathbb{S}^1$-template in $H_\hop$, we need to construct a subdivision cutting the circle wherever the corresponding $H_\hop$-limit cycle hits a reset.  First we will do this for $K$, then translate the results to $H$. To do this cutting of the circle, we define $C$ to be the hybrid system with graph $G(C)= G(K)$ where $I_{v_1}^C = \{(x_1, x_2) \in \mathbb{S}^1  \mid x_1 \le 0 \}$ and $I_{v_2}^C = \{(x_1, x_2) \in \mathbb{S}^1  \mid x_1 \geq 0\}$. Let $X_{v_i}^C$ be the restriction of $f^*X$ to $I_{v_i}^C$ for $i = 1,2$.   We define the guard sets $Z_{12} \subset I_{v_1}^C$ and $Z_{21} \subset I_{v_2}^C$ by $Z_{12} = \{(0,1)\}$ and $Z_{21} = \{(0,-1)\}$.  The reset maps are identity inclusions.  Let $q\colon C \to (\mathbb{S}^1, f^*X)$ be the subdivision given by identifying points connected by reset maps, and let $e\colon C \to K$ be the embedding given by the identity map on graphs and $e_{v_i}(x_1, x_2) = (x_1, x_2)$ for $i=1,2$.  Then we have the following commuting square:
\begin{linenomath*}
\[
    \begin{tikzpicture}
        \node (s1) at (0,0) {$(\mathbb{S}^1, f^*X)$};
        \node (k) at (3,2) {$K$};
        \node (l) at (3,0) {$L$};
        \node (c) at (0,2) {$C$};
        
        \draw[right hook->,below] (s1) to node {$f$} (l);
        \draw[->>, right] (k) to node {$p$} (l);
        \draw[->>, left] (c) to node {$q$} (s1);
        \draw[right hook->, above] (c) to node {$e$} (k);
    \end{tikzpicture}
\]
\end{linenomath*}

The upper-left-hand corner of the above square defines a template-anchor pair $\mathbb{S}^1 \xleftarrow{q} C \xrightarrow{e} K$. To construct a corresponding span for $H_\hop$, we will use the fact that $K$ is a double cover of $H_\hop$.  Since one period of the limit cycle in $K$ corresponds to two periods of the limit cycle in $H_\hop$, the template vector field and its subdivision for $K$  must be modified to work for $H_\hop$. 
First we note that just as $K$ is a double cover of $H_\hop$, there exists a hybrid system $S$ such that $C$ is a double cover of $S$.  More precisely, we define $G(S):= G(H_\hop)$, the continuous mode of $S$ to be the same as the mode $v_1$ of $C$ with the same guard, and the reset by $x \mapsto -x$.  By \Cref{eqn:odd}, there exists a two-fold covering semiconjugacy $t\colon C \to S$ given by $t_{v_1}(x) = x$ and $t_{v_2}(x) = -x$. Moreover, since $s \circ e$ is constant on the fibers of $t$ (both on the graph level and on the data associated to graph elements), there exists an embedding $g \colon S \to H_\hop$ such that the following square commutes:
\begin{linenomath*}
\[
    \begin{tikzpicture}
        \node (s1) at (0,0) {$S$};
        \node (k) at (3,2) {$K$};
        \node (l) at (3,0) {$H_\hop$};
        \node (c) at (0,2) {$C$};
        
        \draw[right hook->,below] (s1) to node {$g$} (l);
        \draw[->>, right] (k) to node {$ s$} (l);
        \draw[->>, left] (c) to node {$ t$} (s1);
        \draw[right hook->, above] (c) to node {$e$} (k);
    \end{tikzpicture}
\]
\end{linenomath*}
Since $g(S) = s(e(C))$ and $e(C)$ is an attracting set in $K$, it follows from the definition of $s$ that $g(S)$ is also an attracting set in $H_\hop$.

Lastly, \Cref{eqn:odd} implies that the vector field associated to the continuous mode of $S$ can be smoothly spliced across its reset map.  More precisely, there exists a subdivision $p_S\colon S \to (\mathbb{S}^1, 2_* f ^*(X))$, where $2_*$ denotes the pushforward by the map $2\colon \mathbb{S}^1 \to \mathbb{S}^1$ given by $\theta \mapsto 2\theta$, which is well-defined by \Cref{eqn:odd}.  Thus the span 
\begin{linenomath*}
\[
    (\mathbb{S}^1, 2_*f^*X) \xleftarrow{p_S} S \xrightarrow{g} H_\hop
\]
\end{linenomath*}
defines a template-anchor pair. \qed
\end{example}

\subsection{Hierarchical composition of template-anchor pairs}

We define the hierarchical composition of template-anchor pairs by taking fiber products of subdivisions.  The following series of propositions demonstrate the compatibility of fiber products with hybrid submersions and isolating neighborhoods.   This section concludes with \Cref{thm:ta_comp}, which shows that this notion of hierarchical composition is weakly associative.

\begin{prop}
\label{prop:temp_ench_fiber_products}
Given a hybrid embedding $i\colon K_2 \hookrightarrow H$ and a hybrid submersion $p\colon K_1 \twoheadrightarrow H$ with resulting fiber product 
\begin{linenomath*}
\[
    \begin{tikzpicture}
        \node (k1) at (0,0) {$K_1$};
        \node (k2) at (3,2) {$K_2$};
        \node (h) at (3,0) {$H$,};
        \node (pb) at (0,2) {$K_1 \times_{p,i} K_2$};
        
        \draw +(.5,1) -- +(1,1)  -- +(1,1.5);

        \draw[->>,below] (k1) to node {$p$} (h);
        \draw[right hook->, right] (k2) to node {$i$} (h);
        \draw[->, left] (pb) to node {$\tilde i$} (k1);
        \draw[->, above] (pb) to node {$\tilde p$} (k2);
    \end{tikzpicture}
\]
\end{linenomath*}
the morphism $\tilde p$ is a hybrid submersion, and $\tilde i$ is a hybrid embedding. Moreover, if $p$ is a surjective submersion or subdivision, then so is $\tilde p$.
\end{prop}
\begin{proof}
Since the fiber product constructions for graphs used in \Cref{prop:hybrid_fiber_products} is fundamentally a pair of set-theoretic fiber products, the graph morphisms associated with $\tilde p$ and $\tilde i$ are necessarily epic and monic, respectively. Similarly, for each $v = (v_1, v_2) \in V(K_1 \times_{p,f} K_2)$ and $x \in I^{K_1 \times_{p,i} K_2}_v$ we get the commutative diagram of linear maps
\begin{linenomath*}
\[
    \begin{tikzpicture}
        \node (k1) at (0,0) {$T_{\tilde i(x)}I^{K_1}_{v_1}$};
        \node (k2) at (4,2) {$T_{\tilde p(x)}I^{K_2}_{v_2}$};
        \node (h) at (4,0) {$T_{p_{v_1} \circ \tilde i_{v}(x)}I^{H}_{p(v_1) = f(v_2)}$};
        \node (pr) at (0,2) {$T_{j_v(x)} I^{K_1 \times K_2}_{v = (v_1, v_2)}$};
        \node (pb) at (-2, 4) {$T_x I^{K_1 \times_{p,i} K_2}_{v = (v_1, v_2)}$};
        
        \draw +(.5,1) -- +(1,1)  -- +(1,1.5);
        
        \draw[->>,below] (k1) to node {$Tp_{v_1}$} (h);
        \draw[right hook->, right] (k2) to node {$Ti_{v_2}$} (h);
        \draw[right hook->, above right] (pb) to node {$Tj_v$} (pr);
        \draw[->, left, out=270, in=160] (pb) to node {$T\tilde i_{v}$} (k1);
        \draw[->, above, out=0, in=140] (pb) to node {$T\tilde p_{v}$} (k2);
        \draw[->, above] (pr) to node {$T\pi_2$} (k2);
        \draw[->, left] (pr) to node {$T\pi_1$} (k1);
    \end{tikzpicture}
\]
\end{linenomath*}
which shows that $T\tilde i_v$ and $T\tilde p_v$ are respectively injective and surjective linear maps using the usual argument for fiber products of linear maps. 

Moreover, if $p$ is a surjective submersion, then this same argument applied to the appropriate fiber product shows that $\tilde p$ is a surjective submersion. For a subdivision, the additional property follows using the definition of subdvisions and then the pasting law for fiber products in the following diagram:
\begin{linenomath*}
\[
    \begin{tikzpicture}
        \node (s) at (-3,0) {$\sigma$};
        \node (k1) at (3,0) {$K_1$};
        \node (k2) at (0,-2) {$K_2$};
        \node (h) at (3, -2) {$H$};
        \node (fp) at (0, 0) {$K_1 \times_H K_2$};
        \node (t) at (-3, -2) {$\tau$};
        
        \draw +(.25,-.75) -- +(.75,-.75)  -- +(.75,-.25);

        \draw[->, below] (k2) to node {$i$} (h);
        \draw[->] (fp) to (k1);
        \draw[->, left] (fp) to node {$\tilde p$} (k2);
        \draw[->, above, out=30, in=150] (s) to node {$\widetilde{i \circ \chi}$} (k1);
        \draw[->, dashed] (s) to (fp);
        \draw[->, left] (s) to node {$\xi$} (t);
        \draw[->, right] (k1) to node {$p$} (h);
        \draw[->, below] (t) to node {$\chi$} (k2);
    \end{tikzpicture}
\]
\end{linenomath*}
since both the right and outside squares are fiber products. 
\end{proof}

\begin{lemma}
\label{lem:iso_fiber_products}
Let $K_1, K_2, H \in \hs$ be hybrid systems. Suppose we have a hybrid embedding $i\colon K_2 \hookrightarrow H$ and a hybrid subdivision $p\colon K_1 \twoheadrightarrow H$ with resulting fiber product 
\begin{linenomath*}
\[
    \begin{tikzpicture}
        \node (k1) at (0,0) {$K_1$};
        \node (k2) at (3,2) {$K_2$};
        \node (h) at (3,0) {$H$.};
        \node (pb) at (0,2) {$K_1 \times_H K_2$};
        
        \draw +(.5,1) -- +(1,1)  -- +(1,1.5);
        
        \draw[->>,below] (k1) to node {$p$} (h);
        \draw[right hook->, right] (k2) to node {$i$} (h);
        \draw[->, left] (pb) to node {$\tilde i$} (k1);
        \draw[->, above] (pb) to node {$\tilde p$} (k2);
    \end{tikzpicture}
\]
\end{linenomath*}
If $i(K_2) \subset I(H)$ is an attracting set with isolating neighborhood  $W \subset I(H)$, then $\tilde{i}(K_1 \times_H K_2) \subset I(K_1)$ is an attracting set with isolating neighborhood $\tilde W := p^{-1}(W) \subset I(K_1)$.  Moreover, if $W$ is a trapping region, then so is $\tilde W$.
\end{lemma}
\begin{proof}
To see that $\tilde W$ is an isolating neighborhood for $\tilde{i}(K_1 \times_H K_2)$, we need to check that 
\begin{linenomath*}
\[
    \tilde{i}(K_1 \times_H K_2) = \bigcap_{t > 0} \left\{\chi(t) \mid x \in \tilde W \text{ and } \chi \in \infexc_{K_1}(x)\right\}.
\]
\end{linenomath*}
Suppose that $y \in I(K_1 \times_H K_2)$.  Since $i(\tilde p(y)) \in i(K_2)$, for every $t > 0$ there exists an $x \in W$ and $\chi \in \infexc_H(x)$ such that $i(\tilde p(y)) \in \chi(t)$.   By the commutativity of the fiber product square, we have $p ( \tilde i (y)) = \chi(t)$.  Let $\tilde{\chi}$ be the fiber product of $\chi$ along $p$ as in \Cref{def:hyb_sub}. By \Cref{lem:ss_time}, we have $p^{-1}(\chi(t)) \in \tilde\chi(t)$.  Thus, $\tilde i (y) \in \tilde \chi (t)$.   Since $\tilde \chi \in \infexc_{K_1}(p^{-1}(x))$, we have 
$
\tilde i (y) \in \bigcap_{t > 0} \left\{\chi(t) \mid x \in \tilde W \text{ and } \chi \in \infexc_{K_1}(x)\right\}.
$

Conversely, suppose that
$y \in \bigcap_{t > 0} \left\{\chi(t) \mid x \in \tilde W \text{ and } \chi \in \infexc_{K_1}(x)\right\}.$
We note that $z \in \chi(t)$ for some $x \in \tilde W$, $\chi \in \infexc_{K_1}(x)$, and  $t > 0$ implies that $p(z) \in (p \circ \chi)^*(t)$ and $(p \circ \chi)^* \in \infexc_{H}(p(x))$.  Thus,
\begin{linenomath*}
\[
    p(y)  \in \bigcap_{t > 0} \left\{\chi(t) \mid x \in W \text{ and } \chi \in \infexc_{H}(x)\right\}.
\]
\end{linenomath*}
Since $W$ is an isolating neighborhood for $i(K_2)$, it follows that $p(y) \in i(K_2)$.    Let $x \in I(K_2)$ be the unique point such that $i(x) = p(y)$.  Then because $I(K_1 \times_H K_2) = I(K_1) \times_{I(H)} I(K_2)$ in $\sets$, we get an element $z \in I(K_1 \times_H K_2)$ such that $\tilde i (z) = y$ and $\tilde p (z) = x$.  Thus, $y \in \tilde{i} (K_1 \times_H K_2)$.

Now suppose that $W$ is a trapping region. To see that $\tilde{W}$ is also a trapping region, we need to check the three conditions of \Cref{def:trapping}.

Condition (1) follows from the fact that a refinement of an infinite hybrid time trajectory is still infinite.  

For condition (2), let $x \in \tilde{W}$ and $\chi \in \infexc_{K_1}(x)$.  Suppose there exists a $t \geq 0$ such that $y \in \chi(t) \setminus W$.  Then $\tilde p(y) \not \in \tilde p(\tilde W) = W$, where the last equality follows from the fact that $I(p)$ is surjective.  Thus $(\tilde p \circ \chi)^*(t) \not\subset W$, a contradiction.

For condition (3), let $T>0$ be the lower bound for the trapping region $W$.  We claim that $T$ works for $\tilde W$ as well.  Suppose not.  Then there exists $t > T$ and $y \in \chi(t) \setminus \tilde W^\circ$.  Thus there exists an open set $U \subset W$ such that $y \not \in U$, hence $p(y) \not\in \tilde p(U) \subset W$.  The submersion condition for $p$ implies that $I(\tilde p)$ is an open map, hence $\tilde p(U)$ is open, so $\tilde p(y) \not\in W^\circ$.  Since $\tilde p(y) \in (\tilde p \circ \chi)^*(t)$, this is a contradiction.
\end{proof}

\begin{lemma}
\label{lem:atat}
Let $A \overset{i}{\hookrightarrow} B \overset{j}{\hookrightarrow} C$ be hybrid embeddings.  Suppose that $i(A)$ is an attracting set in $B$, and $j(B)$ is an attracting set in $C$.   Then $(j \circ i)(A)$ is an isolated invariant  set in $C$.
\end{lemma}
\begin{proof}
Let $W_A \subset B$ be an isolating neighborhood for $A$ in $B$, and let $W_B \subset C$ be an isolating neighborhood for $B$ in $C$.  Since $i$ is an embedding, it induces a homeomorphism between $A$ and $i(A)$.  In particular, there exists an open set $U \subset I(B)$ such that $U \cap A = W_A^\circ$.    Then $U \cap W_B$ is an isolating neighborhood for $A$ in $C$. 
\end{proof}

\begin{thm}
\label{thm:ta_comp}
Template-anchor pairs are weakly associatively composable. 
\end{thm}
\begin{proof}
Using \Cref{prop:temp_ench_fiber_products}, we can (weakly, since fiber products are only unique up to isomorphism) associatively compose two compatible template anchor pairs $T_1 \xleftarrow{p_1} S_1 \xrightarrow{i_1} A_1$ and $A_1 \xleftarrow{p_2} S_2 \xrightarrow{i_2} A_2$ to get a span  $T_1 \xleftarrow{p} S \xrightarrow{i} A_2$ as in this diagram:
\begin{linenomath*}
\[
    \begin{tikzpicture}
        \node (t1) at (0,0) {$T_1$};
        \node (s1) at (1.5,1.5) {$S_1$};
        \node (a1) at (3,0) {$A_1$};
        \node (s2) at (4.5,1.5) {$S_2$};
        \node (a2) at (6,0) {$A_2$};
        \node (fp) at (3,3) {$S = S_1 \times_{i_1, p_2} S_2$};
        
        \draw +(2.5,2) -- +(3,1.5)  -- +(3.5,2);

        \draw[->>, above left] (s1) to node {$p_1$} (t1);
        \draw[right hook->, below left] (s1) to node {$i_1$} (a1);
        \draw[->>, below right] (s2) to node {$p_2$} (a1);
        \draw[right hook->, above right] (s2) to node {$i_2$} (a2);
        \draw[->>, above left] (fp) to node {$\tilde p_2$} (s1);
        \draw[right hook->, above right] (fp) to node {$\tilde i_1$} (s2);
        \draw[->>, left, out=190, in=100] (fp) to node {$p = p_1 \circ \tilde p_2$} (t1);
        \draw[right hook->, out=350, in=80,right] (fp) to node {$i = i_2 \circ \tilde i_1$} (a2); 
    \end{tikzpicture}
\]
\end{linenomath*}
It remains to show that $T_1 \xleftarrow{p} S \xrightarrow{i} A_2$ is a template-anchor pair. By \Cref{prop:temp_ench_fiber_products} and \Cref{lem:subd_compose}, the semiconjugacy $p$ is a hybrid subdivision.  Since the composition of hybrid embeddings is a hybrid embedding, the semiconjugacy $i$ is a hybrid embedding.

Let $U \subset I(A_2)$ be a hybrid trapping region with trapped attracting set $S_2$. Since $S_1$ is attracting in $A_1$, \Cref{lem:iso_fiber_products} implies that $S$ is an attracting set in $S_2$.  It follows from \Cref{lem:atat} that $S$ is an attracting set in $A_2$.
\end{proof}

\section{Sequential composition}  \label{sec:sequential}

The aim of this section is to define sequential composition of a well-behaved class of ``directed'' hybrid systems.  The key condition in the definition of a directed system below is the $(\varepsilon,T)$-chain condition, which, roughly speaking, says that every generalized trajectory of a directed system ends up in its final subsystem.  The corresponding measure-theoretic condition is that almost all points flow into the final subsystem.  Alternatively, there is a topological notion of directed system: an open, dense set flows into the final system.  However, as \Cref{ex:directed} demonstrates, neither the measure-theoretic nor the topological condition behaves well with respect to sequential composition, whereas $(\varepsilon,T)$-chains do.

\subsection{\texorpdfstring{$(\varepsilon,T)$-chains}{(epsilon,T)-chains}}
We now define hybrid $(\varepsilon,T)$-chains, a notion of generalized execution adapted from \cite{conley1978isolated}.  To define such chains for $\varepsilon > 0$, we need the additional data of an extended metric.

\begin{defn}
A {\bf metric hybrid system} is a pair $(H, \dist)$ consisting of a hybrid system $H$ and an extended metric $\dist$ on $I(H)$ compatible with the topology.
\end{defn}

As motivation for the following definition, we point out its similarity to our notion of an $H$-execution (\Cref{def:execution}) once the data of the defining semiconjugacy $\tau \to H$ is completely spelled out.  

\begin{defn}
\label{def:chain}
 Let $(H,d)$ be a metric hybrid system, and $\varepsilon, T \in [0,\infty]$.  An {\bf $\bm{(\varepsilon, T)}$-chain}  in $H$ is a triple $\chi = (\tau, \nu, \varphi)$ where 
\begin{itemize}
\item $\tau = ((\tau_j)_{j=0}^N, c_\tau)$ is a hybrid time trajectory;
\item $\nu = (\nu_j)_{j = 0}^{N-1} \subset V(H)$ is a sequence of vertices;
\item $\varphi = (\varphi_j)_{j = 0}^{N-1}$ is a sequence of smooth maps $\varphi_j\colon I_{v_j}^\tau \to I_{\nu_j}^H$
\end{itemize}
such that
\begin{enumerate}[(i)]
  \item for each $0 \le j < N $,
    \begin{itemize}
      \item $\varphi_j([\tau_{j}, \tau_{j+1})) \subset F_{\nu_j}$; and
      \item the restriction $\varphi_j|_{[\tau_{j}, \tau_{j+1})}\colon [\tau_{j}, \tau_{j+1}) \to F_{\nu_j}$ is an integral curve for the vector field $X^H_{\nu_j}$;
    \end{itemize}
  \item for each $1 \le j < N-1$, there is an element $u_j \in E(H)\sqcup V(H)$ with $\src(u_j) = \nu_j$ and $\tgt(u_j) = \nu_{j+1}$ such that
    \begin{itemize}
        \item if $u_j \in E(H)$, then $\varphi_{j-1}(\tau_j) \in Z_{u_j}$
        \item if $u_j \in V(H)$, then $\varphi_{j-1}(\tau_j) \in F_{u_j}$
        \item letting $r_{u_j}\colon F_{u_j} \to I_{u_j}$ be the inclusion if $u_j \in V(H)$,
        \begin{linenomath*}
        \[
        \dist(\varphi_{j+1}(\tau_{j+1}), r_{u_j}(\varphi_j(\tau_{j+1}))) \le \varepsilon, 
        \]
        \end{linenomath*}
    \end{itemize}
    \item if $\tau'_k$ is the subsequence of $\tau_j$ for which $j = 0$ or $u_j \in V(H)$, then
    \begin{linenomath*}
    \[
        \tau'_{k} - \tau'_{k-1} \geq T
    \]
    \end{linenomath*}
    for all $k \geq 1$.
\end{enumerate}
We say that $\varphi_0(0)$ is the {\bf starting point} of $\chi$. We denote the set of $(\varepsilon, T)$-chains in $H$ (with starting point $x$) by $\conleyet_H$ and $\conleyet_H(x)$, respectively. If $\tau$ has an endpoint, we say that $\chi$ is a chain from $\varphi_0(0)$ to $\varphi_{N-1}(\tau)$.  We denote the set of $(\varepsilon, T)$-chains from $x$ to $y$ by $\conleyet_H(x,y)$.

For $\chi_x = (\tau, \nu, \varphi) \in  \conleyet_H(x)$ and $t > 0$, we write
\begin{linenomath*}
\[
    \chi_x(t) = \{\varphi_j(t) \mid 0 \le j < N \text{ and } \tau_j \le t \le \tau_{j+1} \}.
\]
\end{linenomath*}
We will sometimes write $\chi^t_x$ for $\chi_x(t)$ to avoid a proliferation of parentheses. We emphasize that $\chi_x^{\tau_j}$ is possibly sequence-valued (i.e., a totally ordered countable set of cardinality greater than one) if $\tau_j$ is a jump time, but that if $t \neq \tau_j$ for any $j$, then $\chi_x^t$ is a single point in $I(H)$.    
\end{defn}

A basic property of $(\varepsilon,T)$-chains is their compatibility with semiconjugacy. More precisely, every hybrid semiconjugacy $\alpha\colon H \to K$ between metric hybrid systems induces a set-map $\hat{\alpha}\colon \conley_H^{\varepsilon,T} \to \conley_K^{\omega(\varepsilon),T}$, where $\omega\colon[0, \infty] \to [0,\infty]$ is the modulus of continuity of $I(\alpha)$,  given by 
\begin{linenomath*}
\[
    \hat{\alpha}(\tau, \nu, \phi) = (\tau, (\alpha(\nu_j))_j, (\alpha_{\nu_j} \circ \varphi_j)_j).
\]
\end{linenomath*}

Specializing to $\varepsilon = 0$, \Cref{def:chain} becomes independent of the extended metric $\dist$.  One expects executions to correspond to $(0,0)$-chains, and fundamental executions to correspond to $(0,\infty)$-chains (which correspond to executions in the sense of \cite{lygeros:executions}). Modulo technical details due to the formulation using smooth sets, this is the case. Indeed, by comparing Definitions  \ref{def:htt} and \ref{def:chain}, we get a version of the main theorem of \cite{lerman}:

\begin{thm}
\label{thm:push_forward_executions}
For any metric hybrid system $H$, there exists a surjection
\begin{linenomath*}
\[
    e_H\colon \bigcup_{\tau} ~\Hom_{\hs}(\tau, H) \to \conley^{0,0}_H
\]
\end{linenomath*}
given by $(\chi\colon \tau \to H) \mapsto (\tau, (\chi(v_j))_j, (\chi_{v_j})_j)$.   Moreover, the following hold:
\begin{itemize}
    \item $\hat{\alpha} \circ e_H = e_K \circ \alpha$ for any semiconjugacy $\alpha\colon H \to K$, where $\alpha$ acts on $ \bigcup_{\tau} ~\Hom_{\hs}(\tau, H)$ by postcomposition
    \item for any executions $\chi, \chi' \in \exc_H$ , we have $e_H(\chi) = e_H(\chi')$ if and only if $I(\chi) = I(\chi')$ 
    \item an execution $\chi \in \exc_H$ is fundamental if and only if $e_H(\chi) \in \conley^{0,\infty}_H$
\end{itemize}
\end{thm}

The map $e_H$  will typically fail to be injective. For example, consider the hybrid time trajectory $\tau = ((0,0,0), \text{closed})$, {\em i.e.} a single one-point guard set followed by a one-point flow set.  Let $z$ denote the unique point in $Z_{e_0}^\tau = I_{v_0}^\tau$.   A hybrid semiconjugacy $\chi\colon \tau \to H$ depends on a choice of germ-equivalence class of smooth extensions $\tilde \chi_{v_0}$ of a map $z \mapsto \chi_{v_0}(z)$ to neighborhoods of $z$ in $M_{v_0}^\tau := \mathbb{R}$.  If $\chi_{v_0}(z)$ is in a guard set, there are no other constraints on this choice of extension, so there will generally be many such germ-equivalence classes, each defining a different semiconjugacy $\tau \to H$ with the same corresponding $(0, \infty)$-chain.   

\subsection{Directed systems}  With our notion of generalized trajectory in hand, we can now define the building blocks for sequential composition.

\begin{defn} \label{def:directed}
Let $\init{H}, \fin{H}$ be metric hybrid systems.  A {\bf directed hybrid system} $H \colon \init{H} \leadsto \fin{H}$ is a tuple $(H, \init{\eta}, \fin{\eta})$ consisting of 
\begin{enumerate}
    \item[(1)] a metric hybrid system $H$
    \item[(2)] a hybrid embedding $\init{\eta} \colon \init{H} \to H$ 
    \item[(3)] a hybrid embedding $\fin{\eta}\colon \fin{H} \to H$ such that each component $(\fin{\eta})_v$ is a diffeomorphism, and $G(\fin H)$ is a sink in $G(H)$
\end{enumerate} 
such that for all $\varepsilon, T >0$ and $x \in I(H)$, there exists an $(\varepsilon,T)$-chain from $x$ to some $y \in I(\fin{H})$.  
\end{defn}

We note that if $I(H)$ is compact, then \Cref{def:directed} is independent of the choice of extended metric on $I(H)$. 

As described above, the idea behind a directed system is that trajectories from the domain flow (in a hybrid sense) to the codomain. At first glance, it would seem that we could use executions for this purpose. The following example demonstrates that using executions creates a problem for sequential composition that is resolved by using the more general notion of  $(\epsilon,T)$-chains in place of executions.

\begin{example}
\label{ex:directed}
Let $H$ be the hybrid system given by $G(H) = \overset{v}{\bullet} \xrightarrow{e} \overset{w}{\bullet}$ where 
\begin{itemize}
    \item $M_v = \mathbb{R} = I_v$;
    \item $(M_w, X_w) = (*, 0)$ and $F_v = *$; 
    \item $Z_e = \mathbb{R}$.
\end{itemize}
Let $K$ be the hybrid system given by $G(K) = \overset{y}{\bullet} \xrightarrow{f} \overset{z}{\bullet}$ with
\begin{itemize}
    \item $(M_y, X_y) = (\mathbb{R}, x \frac{d}{dx})$, $I_y = [0,1]$, and $F_y = [0,1)$;
    \item $(M_z, X_z) = (*, 0)$ and $F_y = *$; 
    \item $Z_f = \{1\}$.
\end{itemize}
Then $H$ defines a directed system $H \colon H|_v \leadsto (*,0)$, and $K$ defines a directed system $K \colon (*,0) \leadsto K|_z$ where the embedding $\init\eta^K\colon (*,0) \to K$ maps to the equilibrium point $0 \in I_y$.  Moreover, every infinite execution in $H$ ends up in its final subsystem, and every infinite execution in $K$ with starting point different from $0$ ends up in its final subsystem.  The sequential composition $H$ followed by $K$ (as defined in \Cref{thm:double_cat}) has the graph $\overset{v}\bullet \xrightarrow{e} \overset{y}{\bullet} \xrightarrow{f} \overset{z}{\bullet}$ with the same data as above except the image of $r_e$ becomes the equilibrium point $0 \in I_y$.  Thus, there is no execution from $I_v$ to $I_z$, but there are $(\epsilon,T)$-chains from every point of $I_v$ that end in $I_z$.  
\end{example}

The following proposition and corollary show that the $(\varepsilon,T)$-chain condition of \Cref{def:directed} is strictly weaker than the analogous topological and measure-theoretic conditions.

\begin{prop}
\label{prop:top_directed}
Suppose the triple  $(H, \init\eta, \fin\eta)$ satisfies all conditions of \Cref{def:directed} except for the $(\varepsilon,T)$-chain condition.  Further suppose $H$ is nonblocking and  that a dense subset of $I(D)$ flows into $\fin\eta(\fin{H})$.
Then $H$ is a directed system.
\end{prop}
\begin{proof}
Suppose $H$ is not a directed system. Then there exist $\varepsilon, T > 0$ and a point $x \in I(H)$ such that there exists no $(\varepsilon,T)$-chain from $x$ into $\fin\eta(\fin{H})$. Let $\chi \in \infexc_H(x)$.  Let $y \in I(H)$ be defined via one of the following two cases. 
If $\chi$ has no jumps, then let $y = \chi(T)$.  If $\chi$ has a jump at some time $t$, let $y \in I(H)$ be any element of $\chi(t)$ other than the first.  Then in either case there is no execution from $z$ into $\fin\eta(\fin{H})$ for each $z \in B_\varepsilon(y)$, a contradiction.
\end{proof}

Recall that for any manifold $M$, we say that a subset $B \subset M$ has measure zero if its intersection with any smooth chart has measure zero \cite{lee:smooth_manifolds}.   We  use the phrase ``almost all'' in the usual sense to refer to conditions that hold except on a set of measure zero.   Since a set of measure zero cannot contain an open ball, we have the following corollary to \Cref{prop:top_directed}.

\begin{cor}
\label{cor:meas_directed}
Suppose the triple  $(H,  \init{\eta}, \fin{\eta})$ satisfies all conditions of \Cref{def:directed} except for the $(\varepsilon,T)$-chain condition.  Further suppose $H$ is nonblocking and that almost all points of $I(H)$ flow into $\fin\eta(\fin{H})$.
Then $H$ is a directed system.
\end{cor}

\subsection{A navigation example}
To give an example of a directed system, we recall the following definitions from  \cite{Arslan_Koditschek_2018}. Let $\mathcal{W} \subset \mathbb{R}^2$ be a closed convex set.  Let $\mathcal{O} = \{O_1, \ldots, O_m\}$ be a finite set of nonmoving obstacles $O_i$, where each $O_i \subset \mathcal{W}$ is a convex set with twice differentiable boundary.  

We assume a disk-shaped robot of radius $r$ operating in this environment.  The free space for the robot is the set
\begin{linenomath*}
\[
    \mathcal{F} = \mathcal{W} \setminus \bigcup_i B_r(O_i),
\]
\end{linenomath*}
where $B_r(O_i)$ denotes the union of the open balls of radius $r$ around each point of $O_i$.
We further assume the robot has complete knowledge of the environment within the open ball of radius $R$ of its center.  If an obstacle intersects this perceptual disk, we say that it is visible (to the robot).

Given a global goal location $x^* \in \mathcal{W}$ and a set $A \subset W$,  the projection of the goal onto $A$ is
\begin{linenomath*}
\[
    \Pi_A(x^*) := \argmin_{a \in A}  \dist(a, x^*).
\]
\end{linenomath*}
If the robot's center lies at $x \in \mathcal{W}$, we define $V_x \subset W$ to be the Voronoi cell containing $x$ given by the maximum margin separating hyperplanes between the robot and all visible obstacles.

In \cite{Arslan_Koditschek_2018}, it was shown that the ``move-to-projected-goal law''
\begin{linenomath*}
\[
    u(x) = x - \Pi_{V_x \cap \mathcal{F}}(x^*),
\]
\end{linenomath*}
 leaves the free space $\mathcal{F}$ positively invariant.  Moreover, under a natural well-separated condition and a curvature condition on obstacle boundaries, the goal point $x^*$ is an asymptotically stable equilibrium whose basin includes almost all points of $\mathcal{F}$.

More formally, let $H \in \hempty$ be the following hybrid system.  The vertices in $V(H)$ are subsets $S \subset \{1, \ldots, m\}$. Given vertices $S \neq T$,  we define a reset  $e_{ST} \in E(H)$ from $S$ to $T$ if either $S \subsetneq T$ and $|S| = |T| -1$ or vice versa.

To any continuous mode $S$,  we associate the manifold  $M_S := \mathbb R^2$ and the active set $I_S$ to be the closure of the convex, open set $P_S := \bigcap_{s \in S} B_R(O_s) \cap \mathcal{F}$.  We define the vector field $X_S$ to be the gradient of $u$.  To associate resets and guards to a reset $e_{ST}$, there are two cases: $S \subset T$ or $T \subset S$.  
If $S \subset T$ (so $I_T \subset I_S$), then we let $Z_{ST}$ be the set of points $x \in \partial I_T$ such that for every $\varepsilon > 0$ there exists $0 < t < \varepsilon$ such that $\gamma_x(t) \in P_T$, where $\gamma_x$ is an integral curve starting at $x$. If $T \subset S$, then we define $Z_{ST} := (\partial I_S \cap I_T) \setminus Z_{TS}$. In either case, we define the reset map $r_{ST}\colon Z_{ST} \to I_T$ to be the identity map.

The resulting hybrid system $H$ is nonblocking and deterministic, assuming it has no fundamental Zeno executions.  Since the basin of $x^*$ contains almost all points of $\mathcal F$, it follows from \Cref{cor:meas_directed} that $H$ defines a directed hybrid system from $H \leadsto H|_{S^*}$, where $S^* \subset V(H)$ is the perceptual region containing the goal location $x^*$.

\subsection{A double category of directed systems}
To provide a framework for simultaneous reasoning about sequential composition and hybrid semiconjugacy, we construct a double category of directed hybrid systems, a subclass of hybrid systems amenable to sequential composition. An idea first introduced by Ehresmann \cite{ehresmann1963categories}, a double category $\catname C$ is a category internal to $\catname{Cat}$, the category of (small) categories. More concretely, this means that $\catname C$ is given by collections of objects, vertical morphisms, horizontal morphisms and $2$-cells that fit into squares
\begin{linenomath*}
\[
    \begin{tikzpicture}
        \node (x) at (0,0) {$X$};
        \node (y) at (2,0) {$Y$};
        \node (z) at (0,-2) {$Z$};
        \node (w) at (2, -2) {$W$};
        
        \draw[->, directed arrow, above] (x) to node {$\alpha$} (y);
        \draw[->, left] (x) to node {$f$} (z);
        \draw[->, right] (y) to node {$g$} (w);
        \draw[->, directed arrow, below] (z) to node {$\beta$} (w);
        \draw[-implies, double distance=3pt, left] (1,-.5) to node {$\varphi$} (1,-1.5);
    \end{tikzpicture}
\] 
\end{linenomath*}
which can be associatively composed both vertically and horizontally, along with vertical identities and horizontal units all satisfying standard coherence axioms (see \cite{grandis1999limits, shulman-monoidal_bicategories} for more details). As in \cite{shulman-monoidal_bicategories}, we are primarily interested in what are sometimes called pseudo-double categories, where the vertical composition gives a strict category, but horizontally is only weakly associative and unital. We will drop the ``pseudo'' prefix and refer to these as simply double categories. Axiomatically, a straightforward way to describe the data for a double category is as follows.

\begin{defn}
A double category $\catname C$ is given by a pair of categories $\catname C_0, \catname C_1$ together with unit and source/target functors
\begin{linenomath*}
\begin{gather*}
    U\colon \catname C_0 \to \catname C_1\\
    S,T\colon \catname C_1 \to \catname C_0
\end{gather*}
\end{linenomath*}
and a horizontal composition functor
\begin{linenomath*}
\[
    \odot\colon \catname C_1 \times_{\catname C_0} \catname C_1 \to \catname C_1
\]
\end{linenomath*}
(where the fiber product is over $\catname C_1 \overset{T}{\to} \catname C_0 \overset{S}{\leftarrow} \catname C_1$) satisfying 
\begin{linenomath*}
\begin{gather*}
    S (U_A) = A \\
    T ( U_A) = A 
\end{gather*}
\end{linenomath*} 
for any object $A \in \catname C_0$ (where $U_A$ denotes $U(A)$) and
\begin{linenomath*}
\begin{gather*}
    S(M \odot N) = S N \\
    T(M \odot N) = T M
\end{gather*}
\end{linenomath*}
for any objects $M,N \in \catname C_1$,
and equipped with natural isomorphisms
\begin{linenomath*}
\begin{gather*}
    a\colon (M \odot N) \odot P \to M \odot (N \odot P) \\
    l\colon U_B \odot M \to M \\
    r\colon M \odot U_A \to M
\end{gather*}
\end{linenomath*}
for objects $M, N, P \in \catname{C}_1$.   Furthermore, $S(a), T(a), S(l), T(l), S(r), T(r)$ must be identities, and triangle and pentagon coherence axioms (analogous to the axioms for monoidal categories) must hold for the horizontal composition operator $\odot$  \cite[Sec.~7.1]{grandis1999limits}.

We call the objects of $\catname C_0$ the {\bf objects} of the double category, the morphisms of $\catname C_0$ the {\bf vertical morphisms}, the objects of $\catname C_1$ the {\bf horizontal morphisms}, and the morphisms of $\catname C_1$ the {\bf squares}. As in the diagram above, we will use ``$\to$'' for vertical morphisms, ``$\leadsto$'' for horizontal morphisms, and ``$\implies$'' for the squares. 
\end{defn}

\begin{thm}
\label{thm:double_cat}
Metric hybrid systems (objects);  hybrid semiconjugacies (vertical morphisms); and directed hybrid systems (horizontal morphisms) fit into squares
\begin{linenomath*}
\[
    \begin{tikzpicture}
        \node (hi) at (0,0) {$\init{H}$};
        \node (hf) at (2,0) {$\fin{H}$};
        \node (ki) at (0,-2) {$\init{K}$};
        \node (kf) at (2, -2) {$\fin{K}$,};
        
        \draw[->, directed arrow, above] (hi) to node {$H$} (hf);
        \draw[->, left] (hi) to node {$\init{\alpha}$} (ki);
        \draw[->, right] (hf) to node {$\fin{\alpha}$} (kf);
        \draw[->, below, directed arrow] (ki) to node {$K$} (kf);
        \draw[-implies, double distance=3pt, left] (1,-.5) to node {$\alpha$} (1,-1.5);
    \end{tikzpicture}
\] 
\end{linenomath*}
where $\alpha\colon H \to K$ is a semiconjugacy restricting to $\init{\alpha}$ and $\fin{\alpha}$ on $\init{H}$ and $\fin{H}$ respectively, to form a double category $\dirhs$ of directed hybrid systems.
\end{thm}
\begin{proof}
 We define the vertical category $\dirhs_0$ to be the category of metric hybrid systems with semiconjugacies, where we include empty modes and resets as in the definition of $\hempty$ (\Cref{def:hyb_sem}).\footnote{We note that $\dirhs_0$ is in fact equivalent to $\hempty$ since every smooth manifold admits a Riemannian metric (by locally pulling back the Euclidean metric and using the standard partition of unity argument), and hybrid semiconjugacy ignores the metric structure.}  We define the objects of the category $\dirhs_1$ to be directed systems $H \colon \init{H} \leadsto \fin{H}$ and the morphisms to be semiconjugacies  $\alpha \colon H \to K$ which restrict to semiconjugacies $\init{\alpha} \colon \init{H} \to \init{K}$ and $\fin{\alpha} \colon \fin{H} \to \fin{K}$.
The unit functor $U \colon \dirhs_0 \to \dirhs_1$ is defined on objects to be the directed system $U(H) \colon H \leadsto H$ given by $U(H) = (H, \id_H, \id_H)$, and on morphisms by the identity.  The source and target functors $S,T \colon \dirhs_1 \to \dirhs_0$ are defined on an object $H \colon \init{H} \leadsto \fin{H}$ by 
 $S(H) = \init{H}$ and $T(H) = \fin{H}$ and on morphisms by the 
 restrictions $S(\alpha) = \init{\alpha}$ and  $T(\alpha) = \fin{\alpha}$.

Next we define the horizontal composition functor $\odot\colon  \dirhs_1 \times_{\dirhs_0} \dirhs_1 \to \dirhs_1$.  Given directed hybrid systems $H\colon \init{H} \leadsto K$ and $H'\colon K \leadsto \fin{H}'$, 
 let $H' \odot H \in \dirhs_1$ be the following hybrid system, which intuitively corresponds to prioritizing $H'$ on the overlapping system $K$.

The graph $G(H' \odot H)$ is given by the pushout square 
\begin{linenomath*}
    \[
    \begin{tikzpicture}
        \node (x) at (0,0) {$G(K)$};
        \node (y) at (3,0) {$G(H')$};
        \node (z) at (0,-2) {$G(H)$};
        \node (w) at (3, -2) {$G(H' \odot H)$};

        \draw +(2,-1.5) -- +(2,-1)  -- +(2.5,-1);

        \draw[->, above] (x) to node {$\init{\eta}^{H'}$} (y);
        \draw[->, left] (x) to node {$\fin{\eta}^{H}$} (z);
        \draw[->, right] (y) to node {$g$} (w);
        \draw[->, below] (z) to node {$f$} (w);
    \end{tikzpicture}
    \]
    \end{linenomath*}
which exists because of \Cref{lem:digraph_complete}.

For the continuous modes, we define for each $v \in V(H'\odot H)$
\begin{linenomath*}
    \[ 
    I_v^{H' \odot H} = 
    \begin{cases}
         I_{g^{-1}(v)}^{H'} & v \in g(V(H')) \\
         I_{f^{-1}(v)}^{H} & \text{otherwise}
    \end{cases}
    \]
    \end{linenomath*}
    and
    \begin{linenomath*}
    \[ F_v^{H' \odot H} = 
    \begin{cases}
         F_{g^{-1}(v)}^{H'} & v \in g(V( H')) \\
         F_{f^{-1}(v)}^{H} & \text{otherwise} \qquad .
    \end{cases}
    \]
    \end{linenomath*}
    
Now in order to define the guard sets, we construct for each $w \in V(K)$ the composition $\beta_w = (\init\eta)_w \circ (\fin\eta)_w^{-1}$. Then for $e \in E(H' \odot H)$, we define
\begin{linenomath*}
    \[ 
        Z_e^{H' \odot H} = 
    \begin{cases}
         Z_{g^{-1}(e)}^{H'} &  e \in g(E(H')) \\
         \beta_{\src(u)}\left(Z_{f^{-1}(e)}^{H}\right) \setminus Z_{\init{\eta}(\src(u))}^{H'} & e  = f(\fin{\eta}^H(u)) \text{ for } u \in E(K) \setminus \init{\eta}^{-1}(E(H')) \\
         Z_{f^{-1}(e)}^{H} & \text{otherwise}
    \end{cases},
    \]
    \end{linenomath*}
    and 
    \begin{linenomath*}
    \[ 
    r_e^{H} :=
    \begin{cases}
     r_{g(e)} & e \in g(E( H')) \\
     \beta_{\tgt(u)} \circ r_{f(e)} \circ \beta_{\src(u)}^{-1}
     & e  = f(\fin{\eta}^H(u)) \text{ for } u \in E(K) \setminus \init{\eta}^{-1}(E(H')) \\
     r_{f(e)} & \text{otherwise}
    \end{cases}
    \]
    \end{linenomath*}
    
Lastly, we define the initial and final maps by 
\begin{linenomath*}
\[
(\init{\eta}^{H' \odot H})_v := \begin{cases}
\beta_w \circ (\init{\eta}^{H_1})_v  & v = f(\fin\eta(w)) \text {for some } w \in V(K) \\
(\init{\eta}^{H_1})_v & \text{otherwise}
\end{cases}
\]
\end{linenomath*}
and  $\fin{\eta}^{H' \odot H} :=\fin{\eta}^{H'}$.
    
We call $H' \odot H := (H' \odot H, \init{\eta}^H, \fin{\eta}^H)$ the {\bf sequential composition} of $H'$ with $H$.  To check the $(\varepsilon,T)$-chain condition, let $\varepsilon, T > 0$, and $x \in I(H') \odot I(H)$.   Then either $x$ lies in  either $I(H')$ or $I(H)$.  If $x \in I(H')$, then there exists an $(\varepsilon,T)$-chain from $x$ to some $y \in I(Z)$ since $H'$ is directed.  If $x \in I(H)$, then there exists and $(\varepsilon,T)$-chain from $x$ to some $y \in I(K)$ since $H$ is directed.  Since $H'$ is directed, there exists an $(\varepsilon,T)$-chain from $y$ to some $z \in I(\fin H')$.  Because $(\varepsilon,T)$-chains compose, this implies that there exists an $(\varepsilon, T)$-chain from $x$ to $z$.

On 2-cells, the functor $\odot$ is defined by prioritizing its first argument in an analogous way to its definition on objects.  The weak associativity of $\odot$ follows from the facts that (i) constructing pushouts of the underlying graphs is weakly associative and (ii) the data assigned to vertices and edges of a sequential composition $H' \odot H$ is defined by prioritizing $H'$ over $H$, hence the corresponding data of a string of sequential compositions is completely determined by order (prioritizing left over right).   The left and right unitors are the isomorphisms induced by the isomorphisms of the underlying sets.  
\end{proof}

This notion of sequential composition is clearly compatible with our determinism and nonblocking conditions:
\begin{prop}
If $H\colon X \leadsto Y$ and $H'\colon Y \leadsto Z$ are deterministic (nonblocking) directed systems, then $H' \odot H$ is deterministic (nonblocking).
\end{prop}

We also have the following adaptation of \Cref{prop:cartesian} to the directed system setting.

\begin{prop}
\label{prop:double_cartesian}
The vertical category $\dirhs_1$ is cartesian and cocartesian.  That is, if $H_i\colon \init{(H_i)} \leadsto \fin{(H_i)}$ for $i = 1,2$ are directed systems,
then $H_1 \times H_2$ and $H_1 \sqcup H_2$ are also directed systems.  Moreover, these systems define horizontal morphisms 
$$H_1 \times H_2 \colon \init{(H_1)} \times \init{(H_1)} \leadsto \fin{(H_1)} \times \fin{(H_1)}$$ and  
$$H_1 \sqcup H_2 \colon \init{(H_1)} \sqcup \init{(H_1)} \leadsto \fin{(H_1)} \sqcup \fin{(H_1)}.$$
\end{prop}
\begin{proof}
Let $H_i\colon \init{(H_i)} \leadsto \fin{(H_i)}$ for $i = 1,2$ be directed systems.  Let $H_1 \times H_2 := ({H_1} \times {H_2}, \init{\eta}^{H_1} \times \init{\eta}^{H_2}, \fin{\eta}^{H_1} \times \fin{\eta}^{H_2})$.  

To verify the $(\varepsilon, T)$-chain condition of \Cref{def:directed} for $H_1 \times H_2$,  let $x = (x_1, x_2) \in I(H_1 \times H_2)$,  and $\varepsilon, T > 0$.   Let $\chi_1 \in \conleyet_{H_1}(x_1, y_1)$ for some $y_1 \in \fin{(H_1)}$.  Let $\chi_2 \in \conley^{0,\infty}_{H_2}(x_2, y_2)$ be a chain with the same total time length as $\chi_1$ .  Then we have $\chi := \chi_1 \times \chi_2 \in \conleyet(x,y)$ where $y = (y_1, y_2)$.  Similarly, there exists a chain $\psi_2 \in \conleyet_{H_2}(y_2,z_2)$ for some $z_2 \in I(\fin{(H_2)})$, and a chain $\psi_1 \in \conley^{0,\infty}(y_1, z_1)$ of same total time length as $\chi_1$.  Then we have $\psi := \psi_1 \times \psi_2 \in \conleyet_{H_1 \times H_2}(y,z)$ where $z = (z_1, z_2)$.  Since the graph $G(\fin{(H_1)}$ is a sink in $G(H_1)$, we have $z \in \fin{(H_1 \times H_2)}$.  The concatenation of $\chi$ with $\psi$ defines an $(\varepsilon,T)$-chain from $x$ to $z$.

We can then define the $2$-cell projections $\proj_i\colon H_1\times H_2 \to H_i$ by setting $\proj_i(v_1, v_2) = v_i$ for all $(v_1, v_2) \in V(\graph G(H_1 \times H_2))$ and $\proj_i(p_1, p_2) = p_i$ for all $(p_1, p_2) \in E(\graph G(H_1 \times H_2))$. Additionally, for each $v = (v_1, v_2) \in V(\graph G(H_1 \times H_2)$, we can assign the projection $\pi_i\colon I_{v_1} \times I_{v_2} \to I_{v_i}$ which is a semiconjugacy. 

Now suppose that $K$ is another system with $2$-cells
\begin{linenomath*}
\[
    \begin{tikzpicture}
        \node (x) at (0,0) {$\init{K}$};
        \node (y) at (3,0) {$\fin{K}$};
        \node (z) at (0,-2) {$\init{(H_1)}$};
        \node (w) at (3, -2) {$\fin{(H_1)}$};
        
        \draw[->, above] (x) to node {$K$} (y);
        \draw[->, left] (x) to node {$s_\alpha^1$} (z);
        \draw[->, right] (y) to node {$t_\alpha^1$} (w);
        \draw[->, below] (z) to node {$H_1$} (w);
        \draw[-implies, double distance=3pt, left] (1.5,-.5) to node {$\alpha_1$} (1.5,-1.5);
    \end{tikzpicture}
\] 
\end{linenomath*}
and 
\begin{linenomath*}
\[
    \begin{tikzpicture}
        \node (x) at (0,0) {$\init{K}$};
        \node (y) at (3,0) {$\fin{K}$};
        \node (z) at (0,-2) {$\init{(H_2)}$};
        \node (w) at (3, -2) {$\fin{(H_2)}$};
        
        \draw[->, above] (x) to node {$K$} (y);
        \draw[->, left] (x) to node {$s_\alpha^2$} (z);
        \draw[->, right] (y) to node {$t_\alpha^2$} (w);
        \draw[->, below] (z) to node {$H_2$} (w);
        \draw[-implies, double distance=3pt, left] (1.5,-.5) to node {$\alpha_2$} (1.5,-1.5);
    \end{tikzpicture}
\] 
\end{linenomath*}

Then we can define a $2$-cell 
\begin{linenomath*}
\[
    \begin{tikzpicture}
        \node (x) at (0,0) {$\init{K}$};
        \node (y) at (6,0) {$\fin{K}$};
        \node (z) at (0,-2) {$\init{(H_1)} \times \init{(H_2)}$};
        \node (w) at (6, -2) {$\fin{(H_1)} \times \fin{(H_2)}$};
        
        \draw[->, above] (x) to node {$K$} (y);
        \draw[->, left] (x) to node {$s_\beta$} (z);
        \draw[->, right] (y) to node {$t_\beta$} (w);
        \draw[->, below] (z) to node {$H_1 \times H_2$} (w);
        \draw[-implies, double distance=3pt, left] (3,-.5) to node {$\beta$} (3,-1.5);
    \end{tikzpicture}
\] 
\end{linenomath*}
by setting $\beta(v) = (\alpha_1(v), \alpha_2(v))$ for each $v \in V(\graph G(K))$, and $\beta(e) = (\alpha_1(e), \alpha_2(e))$ for each $e \in E(\graph G(K))$. Then $\beta_v = {\alpha_1}_v\times {\alpha_2}_v$ is a semiconjugacy for each $v \in V(\graph G(K))$ and we get squares of the form
\begin{linenomath*}
\[
    \begin{tikzpicture}
        \node (x) at (0,0) {$Z^K_{\src(e)}$};
        \node (y) at (3,0) {$I^K_{\tgt(e)}$};
        \node (z) at (0,-2) {$Z^{H_1 \times H_2}_{\src(\beta(e))}$};
        \node (w) at (3, -2) {$I^{H_1 \times H_2}_{\tgt(\beta(e))}$};
        
        \draw[->, above] (x) to node {$r_e$} (y);
        \draw[->, left] (x) to node {$\beta_{\src(e)}|_{G^K_{\src(e)}}$} (z);
        \draw[->, right] (y) to node {$\beta_{\tgt(e)}$} (w);
        \draw[->, below] (z) to node {$r_{\beta(e)}$} (w);
    \end{tikzpicture},
\]
\end{linenomath*}
which commute because the component squares commute. Since the projections are defined by the usual projections of maps, $\beta$ is the unique $2$-cell such that $\proj_i \circ \beta = \alpha_i$. 

The cocartesian condition follows easily from \Cref{prop:cartesian}.
\end{proof}

\section{Further Directions}
While this paper lays the groundwork for a theory of formal composition of hybrid systems, the development of this theory is wide open.  One important direction is the unification of our constructions with the theory of parallel composition via networks of open hybrid systems \cite{lerman-schmitt:open_hybrid_systems}.  Such a theory will be essential to synthesizing nontrivial parallel compositions of hybrid systems.  An interoperability result allowing us to construct sequential and hierarchical compositions of networks of hybrid systems via the corresponding compositions of their constituents would be particularly interesting.

Another interesting direction is the construction of a more abstract theory of template-anchor pairs. For example, there should be a functor from our (double) category of hybrid systems to a category in which, for example, the smooth system $(\mathbb{S}^1, \frac{d}{d\theta})$ is a template for all hybrid limit cycles.  One  potential approach is to localize at a class of morphisms containing the subdivisions.     Informally, this category would ``quotient out'' by (a subclass of) bisimulative pairs of hybrid systems.  A step in this direction would be to define bisimulation in $\hs$, perhaps using methods similar to \cite{htp:bisimulations, joyal1996bisimulation}.  In a related vein, we would like to study the compatibility of our constructions with hybrid structural stability \cite{simic2001structural}.  

We are also interested in studying the categorical properties of various operations on hybrid systems.  For example, there should be an ``integration'' operation turning a continuous mode into a reset map ({\em e.g.}, turning a flight mode into a reset since the dynamics are dominated by gravity). In particular, we should be able to replace a hybrid system with the discrete system given by the induced map on a Poincar\'e section.  There should also be a ``reset composition'' operation identifying sequential simultaneous resets with a single reset given by their composition. Lastly, we should have a ``gluing'' operation corresponding to quotienting out by a reset map, analogous to a single step in constructing a hybrifold \cite{simic}. 

We would also like to extend our double-categorical framework to encompass the compatibility of template-anchor pairs with sequential composition.   Since directed systems are cospans, the two notions might be expected to fit into an intercategory analogous to the spans of cospans intercategory \cite{grandis2017intercategories, katis1997span,katis2000formalization, albasini2009cospans}. 

On the applications front, we want to explore the type theory associated with our category \cite{barr-wells} and explore the relationship  of this framework to existing work on hybrid system synthesis from specifications in linear temporal logic (LTL) \cite{kress2009temporal}.  Ideally, our ``language"  would provide  physically-grounded symbols and new connectives  for the atomic propositions in an LTL specification.  A step in this direction is to synthesize complete Lyapunov functions \cite{conley1978isolated} for composite hybrid systems using their components. Intuitively, Lyapunov functions generalize the crucial role that energy landscapes must play in the robot ``programming'' that motivate this work  \cite{Koditschek_1989}.  We envision the specification of a control policy on top of the composite system using the local minima of this Lyapunov function as symbols.

\section*{Acknowledgements}
This work was supported by ONR  N000141612817, a Vannevar Bush Faculty Fellowship  held by Koditschek and by UATL 10601110D8Z, a LUCI Fellowship held by Culbertson, both granted by
the Basic Research Office of the US Undersecretary of Defense for Research and Engineering. The authors also gratefully acknowledge helpful conversations with Yuliy Barishnikov, Dan Guralnik, Brendan Fong, Sanjeevi Krishnan, Matthew Kvalheim and Eugene Lerman.

\bibliographystyle{amsalpha}
\bibliography{IEEEabrv,hybrid_sys,kod}

\end{document}